\documentclass[draft,12pt]{amsart}
\usepackage{amssymb, amsmath}
\usepackage{url}
\usepackage{enumerate}

\newtheorem{theorem}{Theorem}[section]
\newtheorem{corollary}[theorem]{Corollary}
\newtheorem{prop}[theorem]{Proposition}
\newtheorem{lemma}[theorem]{Lemma}

\theoremstyle{remark}
\newtheorem{remark}[theorem]{Remark}

\theoremstyle{definition}
\newtheorem{defn}[theorem]{Definition}

\newtheorem{examples}[theorem]{Examples}
\numberwithin{equation}{section}
\numberwithin{theorem}{section}

\newcommand{\fastx}{f\ast\Theta_t(x)}
\newcommand{\fast}{f\ast\Theta_t}
\newcommand{\sqrtpt}{\sqrt{\pi t}}

\newcommand{\balexc}{{\mathcal B}_c}
\newcommand{\balexct}{{\mathcal B}_{c,\tau}}
\newcommand{\acn}{{\mathcal A}^n_c}

\newcommand{\alexc}{{\mathcal A}_c}

\newcommand{\alexct}{{\mathcal A}_{c,\tau}}

\newcommand{\Rbar}{\overline{\R}}
\newcommand{\D}{{\mathcal D}(\R)}
\newcommand{\Dp}{{\mathcal D}'(\R)}

\newcommand{\bv}{{\mathcal BV}}
\newcommand{\intinf}{\int^\infty_{-\infty}}

\newcommand{\N}{{\mathbb N}}
\newcommand{\R}{{\mathbb R}}
\newcommand{\C}{{\mathbb C}}
\newcommand{\fn}{\!:\!}

\newcommand{\linf}{\inf\limits}
\newcommand{\Ibv}{{\mathcal IBV}}
\newcommand{\Ibvn}{{\mathcal IBV}^n}
\newcommand{\nbv}{\mathcal NBV}

\providecommand{\abs}[1]{\lvert#1\rvert}
\providecommand{\norm}[1]{\lVert#1\rVert}

\providecommand{\acany}[1]{{\mathcal A}_c^{#1}}
\providecommand{\acnorm}[1]{\norm{#1}^{(n)}}
\providecommand{\actnorm}[1]{\norm{#1}_{\tau}}
\providecommand{\alexcany}[1]{{\mathcal A}_{c,#1}}

\begin{document}

\date{Preprint January 16, 2015.  To appear in {\it Advances in Pure and Applied Mathematics}.}
\title[Heat equation in Alexiewicz norm]
{The one-dimensional heat equation in the Alexiewicz norm}
\author{Erik Talvila}
\address{Department of Mathematics \& Statistics\\
University of the Fraser Valley\\
Abbotsford, BC Canada V2S 7M8}
\email{Erik.Talvila@ufv.ca}
\thanks{Written while visiting the Department of Mathematics,
University of Arizona.  Thank you for the generous hospitality.}

\begin{abstract}
A distribution on the real line has a continuous primitive integral if
it is the distributional derivative of a function that is continuous
on the extended real line.
The space of
distributions integrable in this sense is a Banach space that includes
all functions integrable in the Lebesgue and Henstock--Kurzweil senses.
The one-dimensional heat equation is considered with initial data that
is integrable in the sense of the continuous primitive integral.
Let $\Theta_t(x)=\exp(-x^2/(4t))/\sqrt{4\pi t}$ be the heat kernel.
With initial
data $f$ that is the distributional derivative of a continuous function,
it is shown that $u_t(x):=u(x,t):=f\ast\Theta_t(x)$  is a classical solution of
the heat equation $u_{11}=u_2$.  The estimate
$\|f\ast\Theta_t\|_\infty\leq\|f\|/\sqrt{\pi t}$ holds.
The Alexiewicz norm is $\|f\|=\sup_I|\int_If|$,
the supremum taken over all intervals.  
The initial data is taken on
in the Alexiewicz norm, $\|u_t-f\|\to 0$ as $t\to 0^+$.  
The solution of the heat equation is unique under the assumptions
that $\|u_t\|$ is bounded and $u_t\to f$ in the Alexiewicz
norm for some integrable $f$.
The heat equation is also considered with initial data that is the
$n$th derivative of a continuous function and in weighted spaces
such that $\int_{-\infty}^\infty f(x)\exp(-ax^2)\,dx$ exists for some $a>0$.  Similar results
are obtained.
\end{abstract}

\maketitle

\section{Introduction}\label{sectionintroduction}
The one-dimensional heat equation is the canonical parabolic 
partial differential equation of second order.  For simple geometries
solutions can be represented explicitly as series or integrals.  It is
well-known that with heat conduction on an infinite rod the solution
is given by a convolution of initial data with the heat kernel.  The
classical problem is:
\begin{align}
&u \in C^2(\R)\times C^1((0,\infty))
\text{ such that}\label{heatclassic}\\
&u_t-u_{xx}=0 \text{ for } (x,t)\in\R\times(0,\infty).
\label{heatpde}
\end{align}
The Gauss--Weierstrass heat kernel is
$\Theta_t(x)=\Theta(x,t)=(4\pi \abs{t})^{-1/2}e^{-x^2/(4t)}$ (defined for
$t\not=0$). The solution
of \eqref{heatclassic}-\eqref{heatpde} is then given by the convolution
of initial data $f$ with the heat kernel
\begin{equation}
\fastx=\intinf f(x-\xi)\Theta_t(\xi)\,d\xi=
\intinf f(\xi)\Theta_t(x-\xi)\,d\xi.\label{convolution}
\end{equation}

There are various ways to impose initial conditions.  If $f$ is bounded
and continuous on $\R$ then $\fastx$ is continuous on $\R\times[0,\infty)$
and $\norm{f\ast\Theta_t-f}_\infty\to 0$ as $t\to 0^+$.
If $f\in L^p(\R)$
for some $1\leq p<\infty$ then $\norm{f\ast\Theta_t-f}_p\to 0$ as
$t\to 0^+$.  For example, see \cite{cannon}, \cite{widderbook}, 
\cite{hirschmanwidder}.

In this paper we allow the integral in \eqref{convolution} to exist
as a continuous primitive integral.
This integration process includes
the Lebesgue and Henstock--Kurzweil integrals with respect to Lebesgue
measure.
An attractive feature of this integral
is that the space of integrable distributions is a
Banach space isometrically isomorphic to the continuous functions on 
the extended real line that vanish at $-\infty$.
These distributions tend to behave much more
like $L^p$ functions than arbitrary distributions and many properties
typical
of functions hold in these spaces, such as integration by parts, continuity
in norm, H\"older inequality, etc.
Whereas results involving arbitrary distributions hold weakly, here we
have results holding in the strong (norm) sense. 

The main idea behind the continuous primitive integral is to define a
class of continuous functions (the primitives) that is a Banach space
under a uniform or weighted uniform norm.  The integrable distributions
are then defined to be the distributional derivatives of the primitives.
Integration is defined via the fundamental theorem of calculus: 
If $F$ is a continuous function on $\Rbar$ then its distributional 
derivative
$F'$ is integrable and $\int_a^b F'=F(b)-F(a)$ for all
$a,b\in\Rbar$.
The space of integrable distributions is a Banach space isometrically 
isomorphic to the space of primitives.  See the following section for
more detail.

We define three classes of
integrable distributions.  First, primitives are the continuous functions
on the extended real line, vanishing at $-\infty$ 
(Section~\ref{sectionalex}).  This integral then
includes the Lebesgue and Henstock--Kurzweil integrals with respect
to Lebesgue measure.  The formula $\int_a^b F'=F(b)-F(a)$ continues
to hold when $F$ is continuous, even if the pointwise derivative
vanishes almost everywhere or fails to exist at any point.
Secondly, we take
higher derivatives of such primitives (Section~\ref{sectionhigher}).  
This includes functions with
algebraic singularities of order $(x-x_0)^{-a}$ at any point $x_0\in\R$,
for all exponents $a>0$.  Finally, we use primitives such that
the weighted integral $\intinf f(x)\exp(-ax^2)\,dx$ exists for some
$a>0$
(Section~\ref{sectionweighted}).  This
includes the $L^p$ spaces with respect to Lebesgue measure and weights 
$\exp(-ax^2)$
($a>0$).  

For each of
these cases we prove theorems of the following type, the respective norm shown
generically as $\norm{\cdot}$.  For $f$ integrable in one of the three
above senses: $f\ast\Theta_t(x)$ is separately analytic in $x$ and $t$;
is a classical solution of the heat equation 
\eqref{heatclassic}-\eqref{heatpde}; satisfies an inequality 
$\norm{f\ast\Theta_t}\leq C_t\norm{f}$ for a sharp constant $C_t$.
The initial data is taken on in the sense that
if $u_t(x)=\fastx$ then
\begin{equation}
\norm{u_t-f}\to 0 \text{ as } t\to 0^+.\label{heatic}
\end{equation}
A uniqueness
theorem is proved under \eqref{heatclassic}-\eqref{heatpde} with the
only additional assumptions $\norm{u_t}$ is bounded and
$\norm{u_t-f}\to 0$ as $t\to 0^+$ for some integrable distribution $f$
(Section~\ref{sectionuniq}).
Pointwise and norm estimates are given for $f\ast\Theta_t$ 
and its derivatives.

An extensive bibliography on classical results appears in \cite{cannon}.  
See also \cite{widderbook}.  For initial data in $L^p$ spaces see
\cite{gustafson} and
\cite{hirschmanwidder}.  The heat equation is considered in weighted
spaces of signed measures in \cite{watson}.  Distributional solutions
are studied in \cite{dautraylions} and  \cite{szmydt}.  These last two
works consider tempered distributions, for which Fourier transform
methods apply.  The initial data is taken on in the distributional
(weak) sense.  An important feature of our results is that all of the
distributions we consider (some tempered, some not tempered) are in Banach
spaces and the initial data is taken on in the strong (norm) sense.

\section{The continuous primitive integral}\label{sectiondistributionalintegrals}

The notation we use for distributions is standard.
The test functions are $\D=C^\infty_c(\R)$.  Distributions are
denoted $\Dp$. 
We  will  usually denote distributional
derivatives by $F'$ and  pointwise derivatives by $F'(t)$.
If $f$ is a locally integrable function in the Lebesgue or 
Henstock--Kurzweil sense we identify $f$ with its distribution
$T_f$ defined by $\langle T_f,\phi\rangle=\intinf f(x)\phi(x)\,dx$ for
$\phi\in\D$.
The results on distributions we use
can  be found in \cite{friedlanderjoshi} or \cite{folland}.

Denote the extended real line by $\Rbar=[-\infty,\infty]$.  We define
$C(\Rbar)$ to be the continuous functions $F$ such that 
$F(\infty)=\lim_{x\to\infty}F(x)$
and $F(-\infty)=\lim_{x\to-\infty}F(x)$ both exist as real numbers.
The space of primitives for the continuous primitive integral is 
$\balexc=\{F\in  
C(\Rbar)\mid F(-\infty)=0\}$.
Note that since $F(-\infty)=0$, $\balexc$ is a Banach space under the norm 
$\norm{F}'_\infty=\sup_{x<y}\abs{F(x)-F(y)}$.  Let 
$
\alexc =\{f\in\Dp\mid f=F' \mbox{ for some }
F\in\balexc\}$. We have made the arbitrary but convenient choice of making
primitives vanish at $-\infty$.  A consequence is that if $f\in\alexc$ then 
it has a unique primitive 
$F\in\balexc$.
Otherwise, the primitives of $f$ would differ by a constant.  The integral is then defined
$\int_a^bf=F(b)-F(a)$ for all $a,b\in\Rbar$.  
Hence, a distribution $f$ has a continuous primitive integral if it is the 
distributional
derivative of a function $F\in\balexc$, i.e., for all $\phi\in\D$ we have
$\langle f,\phi\rangle
=\langle F',\phi\rangle=-\langle F,\phi'\rangle=-\int_{-\infty}^\infty
F(x)\phi'(x)\,dx$.  Since $F$ and $\phi'$ are continuous and $\phi'$ has compact
support, the integral defining the derivative exists as a Riemann integral.  
The Alexiewicz norm of $f\in\alexc$ is
$\norm{f}:=\sup_{x<y}\abs{\int_x^yf}=\sup_{x<y}\abs{F(x)-F(y)}:=\norm{F}'_\infty$ where
$F\in\balexc$ is the primitive of $f$.
Note that $\alexc$ is a Banach space isometrically
isomorphic to $\balexc$.
We get an equivalent norm when we fix
$x=-\infty$. Then $\norm{f}':=\norm{F}_\infty$ and this makes $\alexc$ into a
Banach space that is isometrically isomorphic to $\balexc$ with the uniform norm.  For $F\in\balexc$,
$\norm{F}_\infty\leq\norm{F}'_\infty\leq 2\norm{F}_\infty$.  However, the form
of Alexiewicz norm we are using is somewhat more convenient.
Note that
$L^1$ and the spaces of Henstock--Kurzweil and wide Denjoy integrable
functions are contained in $\alexc$ since these spaces all have
primitives that are continuous \cite{gordon}.  (They are in fact 
dense subspaces of $\alexc$.)  
The function $f(x)=x^{-2}\sin(x^{-3})$ is not in $L^1_{loc}$ but is
in $\alexc$, as can be seen via integration by parts.
If $F\in C(\Rbar)$ and singular ($F'(x)=0$
for almost all $x\in\R$) then the Lebesgue integral of $F'$ exists and
is $\int_E F'(x)\,dx=0$ for each measurable set $E$.  But $F'\in\alexc$ with
continuous
primitive integral $\int_a^bF'=F(b)-F(a)$.  If $F\in C(\Rbar)$ such that
it has a pointwise derivative nowhere then the Lebesgue integral of $F'$
is meaningless but $F'\in\alexc$ with
continuous primitive integral $\int_a^bF'=F(b)-F(a)$.  
The Alexiewicz norm seems to first appear in \cite{alexiewicz}.
The continuous primitive integral has its genesis in the work of 
Mikusi\'nksi and Ostaszewski \cite{pmikusinski};
Bongiorno and  Panchapagesan \cite{bongiornopanchapagesan};
Ang, Schmitt, Vy \cite{ang}; B\"aumer, Lumer and Neubrander 
\cite{baumerlumerneubrander}.
For a detailed 
overview, see \cite{talviladenjoy}.

If $g\fn\R\to\R$ its variation is $Vg=\sup\sum\abs{g(x_i)-g(y_i)}$
where the supremum is taken over all disjoint intervals $(x_i,y_i)\subset\R$.
The functions of bounded variation are denoted $\bv$.  If $g\in\bv$ then it
has limits at infinity and we define $g(\pm\infty)=\lim_{x\to\pm\infty}g(x)$.
Functions of bounded variation form the
multipliers for $\alexc$.  If $f\in\alexc$ with primitive
$F\in\balexc$ and $g\in\bv$ then the integration by 
parts formula is
\begin{equation}
\intinf fg=F(\infty)g(\infty)-\intinf F(x)\,dg(x).\label{generalparts}
\end{equation}
The last integral is a Henstock--Stieltjes integral.  See, for example, \cite{mcleod}.
When $g$ is absolutely continuous the formula simplifies to
the Lebesgue integral
\begin{equation}
\intinf fg=F(\infty)g(\infty)-\intinf F(x)g'(x)\,dx.\label{parts}
\end{equation}
A type of H\"older inequality for $f\in\alexc$ and $g\in\bv$ is
\cite[Lemma~24]{talvilafouriertransform}
\begin{equation}
\left|\intinf fg\right|\leq
\left|\intinf f\right|\inf\abs{g}+\norm{f} Vg\leq \norm{f}\left(\abs{g(\infty)}+Vg\right).
\label{holder}
\end{equation}

A convergence theorem for the continuous primitive integral
\cite[Theorem~22]{talviladenjoy}:
\begin{theorem}\label{proplimit}
Let $f\in\alexc$.  Suppose $\{g_n\}\subset\bv$ such that there is
$M\in\R$ so that for all
$n\in\N$, $Vg_n\leq M$.  If $g_n\to g$ on $\Rbar$ for a function $g\in\bv$
then $\lim_{n\to\infty}\intinf fg_n = \intinf fg$.
\end{theorem}

If $f\in\alexc$ and $g\in\bv$ then the convolution
$f\ast g(x)=\intinf f(x-\xi)g(\xi)\,d\xi$ is well-defined on $\R$.  The
convolution is continuous and $\norm{f\ast g}_\infty\leq
\norm{f}(\norm{g}_\infty +Vg)$.
Properties of the convolution are
proven for the continuous primitive integral in
\cite{talvilaconv}.  It is shown there by a limiting process that
$f\ast g$ also exists for $g\in L^1$ and that $\norm{f\ast g}\leq
\norm{f}\norm{g}_1$.

Three facts about the heat kernel:
\begin{align}
&\norm{\Theta_t}_1=\norm{\Theta_t}=1, \text{ for } t>0 \label{heatnorm}\\
&\Theta_a\ast\Theta_b=\Theta_{a+b}, \text{ provided } 1/a+1/b>0;\label{heatconvolution}\\
&\Theta_a\Theta_b=\frac{\Theta_{ab/(a+b)}}{2\sqrt{\pi}\,\abs{a+b}^{1/2}},
\text{ provided } a\not=0,\ b\not=0,\ a+b\not=0.\label{heatproduct}
\end{align}

\section{Initial data in the Alexiewicz space}\label{sectionalex}

When $f\in\alexc$ the convolution $u(x,t)=f\ast\Theta_t(x)$ provides
a smooth solution of \eqref{heatpde}.  The initial conditions are taken
on in the Alexiewicz norm \eqref{heatic}. The estimates 
$\norm{f\ast\Theta_t}_\infty\leq\norm{f}/(2\sqrtpt)$ and
$\norm{f\ast\Theta_t}\leq\norm{f}$ are shown to be sharp.

Widder \cite{widderbook} has written solutions of the heat
equation as Stieltjes integrals
$\fastx=\intinf \Theta_t(x -\xi)\,dF(\xi)$ where $F\in\bv$ is the 
primitive
of $f$. 
Positive solutions of the heat equation are necessarily
given by such an integral with $F$ increasing \cite[VIII.3]{widderbook}.  Since $\balexc$ contains primitives
that are not of bounded variation, solutions considered in
Theorem~\ref{theoremheatc} below need not be the difference of two 
positive functions.  When $F$ is not absolutely continuous, solutions
can be singular at each
$x\in\R$ as $t\to 0+$ \cite[XIV.8]{widderbook}.  
This corresponds to distributions $f$ which have
no pointwise values.

\begin{theorem}\label{theoremheatc}
Let $f\in\alexc$. Let the primitive of $f$ be $F\in\balexc$.
\begin{enumerate}
\item[(a)] The integrals 
$f\ast\Theta_t(x)=\Theta_t\ast f(x)=F\ast\Theta_t'(x)
=[F\ast\Theta_t]'(x)$ 
exist for
each $x\in\R$ and $t>0$.  
\item[(b)] $f\ast\Theta_t(x)$ is $C^\infty$ for $(x,t)\in\R\times(0,\infty)$.
\item[(c)] The estimate 
$\norm{f\ast\Theta_t}_\infty\leq\norm{f}/(2\sqrt{\pi t})$
is sharp in the sense that the coefficient of $\norm{f}$ cannot be reduced.
\item[(d)] Define the linear operator $\Phi_t\fn\alexc\to C(\Rbar)$ by
$\Phi_t(f)=f\ast\Theta_t$.  Then $\norm{\Phi_t}=1/(2\sqrtpt)$.
\item[(e)] $\lim_{\abs{x}\to\infty}\fastx=0$.
\item[(f)] For each $t>0$, $f\ast\Theta_t\in\alexc$ and the inequality 
$\norm{f\ast\Theta_t}\leq\norm{f}$ is sharp
in the sense that the coefficient of $\norm{f}$ cannot be reduced.
Define the linear operator $\Psi_t\fn\alexc\to \alexc$ by
$\Psi_t(f)=f\ast\Theta_t$.  Then $\norm{\Psi_t}=1$.
\item[(g)] Let $u(x,t)=f\ast\Theta_t(x)$ then $u$ is a solution of 
\eqref{heatclassic}-\eqref{heatpde} and
\eqref{heatic}.
\item[(h)] For each $t>0$ we have $\intinf \fast=
\intinf f$.
\item[(i)] $f\ast\Theta_t$
need not be in any of the $L^p$ spaces
($1\leq p<\infty$).
\item[(j)] For each $t>0$, $f\ast\theta_t(x)$ is real analytic as a function of 
$x\in\R$.  For each $x\in\R$, $f\ast\Theta_t(x)$ is real analytic as a 
function of $t>0$.
\end{enumerate}
\end{theorem}
\begin{proof}
(a) The multipliers for $\alexc$ are the functions of bounded variation.
See \cite{talviladenjoy}.  For each $x\in\R$ and $t>0$ we have
$V_{\xi\in\R}\Theta_t(x-\xi)=1/\sqrt{\pi t}$ so $f\ast\Theta_t(x)$ 
exists for each
$x\in\R$ and $t>0$.  The convolution is commutative since we can
change variables, $\xi\mapsto x-\xi$ \cite[Theorem~11]{talviladenjoy}.
Integration by parts \eqref{parts} establishes the other equalities, except the last
which follows by Taylor's theorem.

(b) By \cite[Theorem~4.1]{talvilaconv} we can differentiate under the
integral sign and this shows $f\ast\Theta$ is $C^\infty$ in
$\R\times(0,\infty)$.

(c) The heat kernel $\xi\mapsto\Theta_t(x-\xi)$ is monotonic on the intervals 
$(-\infty,x]$ and $[x,\infty)$.
By the second mean value theorem for integrals 
\cite[Theorem~26]{talviladenjoy} there are $x_1\leq x\leq x_2$ such that
\begin{eqnarray*}
\fastx    & = &   \Theta_t(\infty)\int_{-\infty}^{x_1}f+
\Theta_t(0)\int_{x_1}^xf+
\Theta_t(0)\int_x^{x_2}f+
\Theta_t(-\infty)\int_{x_2}^{\infty}f\\
 & = &  \frac{1}{2\sqrtpt}\int_{x_1}^{x_2}f.
\end{eqnarray*}
It now follows that $\norm{f\ast\Theta_t}_\infty\leq\norm{f}/(2\sqrtpt)$.
Let $s>0$ and $f=\Theta_s$.  Then $\norm{f}=1$ and
$\norm{\Theta_s\ast\Theta_t}_\infty
=\norm{\Theta_{s+t}}_\infty=1/[2\sqrt{\pi(s+t)}]$.  Letting
$s\to 0$ shows the estimate is sharp.

(d) The operator norm is given by 
$\norm{\Phi_t}=\sup_{\norm{f}=1}\norm{f\ast\Theta_t}_\infty\leq 
1/(2\sqrt{\pi t})$,
by (c).  The example in (c) shows $\norm{\Phi_t}=1/(2\sqrtpt)$.

(e) By \cite[Theorem~2.1]{talvilaconv}, $\lim_{|x|\to \infty}\fastx=
\Theta_t(\infty)\intinf f=0$.

(f) The inequality $\norm{f\ast g}\leq\norm{f}\norm{g}_1$ is proven
for $f\in\alexc$ and $g\in L^1$ in \cite[Theorem~3.4(a)]{talvilaconv}.
This gives
$\norm{f\ast\Theta_t}\leq\norm{f}\norm{\Theta_t}_1=\norm{f}$.
And,
$\norm{\Psi_t}=\sup_{\norm{f}=1}\norm{f\ast\Theta_t}\leq 1$.
We get equality by taking $f=\Theta_s$ for $s>0$.  Then
$\norm{\Psi}\geq \norm{\Theta_s\ast\Theta_t}=\norm{\Theta_{s+t}}_1=1$.

(g) By (b) we can differentiate under the integral sign.  Since
the heat kernel is a solution of the heat equation in $\R\times(0,\infty)$
so is $f\ast\Theta$.
Continuity in the Alexiewicz norm gives \eqref{heatic}.  See
\cite[Theorem~3.4(e)]{talvilaconv}.

(h) Let $-\infty<\alpha<\beta<\infty$.
Using the Fubini theorem in the Appendix to \cite{talvilaconv}
and integrating by parts, we have
\begin{eqnarray*}
\int_\alpha^\beta \fastx\,dx & = & 
\intinf f(\xi)\int_\alpha^\beta\Theta_t(x-\xi)
\,dx\,d\xi\\
 & = & \frac{1}{2\sqrt{\pi t}}
\intinf f(\xi)\int_{\alpha-\xi}^{\beta-\xi}e^{-x^2/(4t)}\,dx\,d\xi\\
 & = & \frac{1}{2\sqrt{\pi t}}\intinf F(\xi)\left[e^{-(\beta-\xi)^2/(4t)} - e^{-(\alpha-\xi)^2/(4t)}
\right]\,d\xi\\
 & = & \frac{1}{2\sqrt{\pi t}}\intinf \left[F(\beta-\xi)
-
F(\alpha-\xi)\right]e^{-\xi^2/(4t)}\,d\xi.
\end{eqnarray*}
Dominated convergence now allows us to take the limits $\alpha\to-\infty$
and $\beta\to\infty$ under the integral to get $\intinf\fast=\intinf f$.
Notice that there are no improper integrals so the $\alpha$ and $\beta$
limits give existence of the continuous primitive integral.  See Hake's theorem
\cite[Theorem~25]{talviladenjoy}.

(i)  Let $s>0$, $a_n=1/\log(n+1)$ and $b_n=2n^2\sqrt{s+t}$.
Define
$f(x)=\sum (-1)^n a_n\Theta_s(x-b_n)$.  If we have $\abs{x}\leq A$ for a real
number $A$ then $2\sqrt{\pi s}\,a_n\Theta_s(x-b_n)\leq \exp(-(b_n-A)^2/(4s))$
for large enough $n$.  The Weierstrass
$M$-test then shows the series defining $f$ converges uniformly on compact
sets and $f\in C(\R)$.  Let $F(x)=\int_{-\infty}^xf$.  To prove
$f\in\alexc$ show $F\in\balexc$.  Clearly, $F(-\infty)=\lim_{x\to-\infty}
F(x)=0$.
And, $F(\infty)=\sum_{n=1}^\infty(-1)^na_n$.
Since $f$ is continuous, we just need to show $\lim_{x\to\infty}F(x)=
F(\infty)$.  Note that
\begin{eqnarray*}
\abs{F(x)-F(\infty)} & = & \left|\sum_{n=1}^\infty(-1)^na_n\int_{x}^\infty
\Theta_s(y-b_n)\,dy\right|\\
 & = & \frac{1}{2}\left|\sum_{n=1}^\infty(-1)^na_n{\rm erfc}
\left(\frac{x-b_n}{2\sqrt{s}}\right)\right|,
\end{eqnarray*}
where the complementary error function is ${\rm erfc}(x)=\frac{2}{\sqrt{\pi}}
\int_x^\infty e^{-y^2}\,dy$.  Suppose $b_m<x\leq b_{m+1}$ for some $m\geq 2$.
If $n\leq m-1$
then $x-b_n>b_m-b_{m-1}=2\sqrt{s+t}\,(2m-1)\to\infty$ as $m\to\infty$.  And,
\begin{equation}
\left|\sum_{n=1}^{m-1}(-1)^na_n{\rm erfc}
\left(\frac{x-b_n}{2\sqrt{s}}\right)
\right|\leq \sum_{n=1}^{m-1}a_n\,{\rm erfc}
\left(\frac{\sqrt{s+t}\,(2m-1)}{\sqrt{s}}\right).\label{erfc1}
\end{equation}
The complementary error function has asymptotic behaviour
\cite[8.254]{gradshteyn}
${\rm erfc}(x)\sim \exp(-x^2)/(\sqrt{\pi}\,x)$ as $x\to\infty$.
Since $\abs{\sum_1^{m-1}a_n}\leq m/\log(2)$ we see that the term in
\eqref{erfc1} has limit $0$ as $m\to\infty$.
And, 
summation by parts and telescoping give
\begin{align*}
&\left|\sum_{n=m}^\infty(-1)^na_n{\rm erfc}
\left(\frac{x-b_n}{2\sqrt{s}}\right)
\right|\\
&\quad  =   \left|\sum_{n=m}^\infty
\left(\sum_{k=m}^n(-1)^ka_k\right)\left[{\rm erfc}
\left(\frac{x-b_{n+1}}{2\sqrt{s}}\right)
-{\rm erfc}\left(\frac{x-b_n}{2\sqrt{s}}\right)\right]
\right|\\
&\quad \leq  2\sup_{N\geq m}\left|\sum_{k=m}^N(-1)^ka_k\right|,
\end{align*}
since $0\leq {\rm erfc}(x)\leq 2$.  This term has limit $0$ as $m\to\infty$.
Hence,
$f\in\alexc$.

The Fubini--Tonelli theorem
shows we can interchange series and integral to get $f\ast\Theta_t(x)
=\sum(-1)^na_n\Theta_{s+t}(x-b_n)$.  Let $m>1$ and suppose $\abs{x-b_m}\leq
2\sqrt{s+t}$ then $\Theta_{s+t}(x-b_m)\geq 1/[2e\sqrt{\pi(s+t)}]$.
Note that for each $n\geq 1$ 
we have $b_{n+1}-b_n=2\sqrt{s+t}\,(2n+1)>2\sqrt{s+t}$.
The sequence $\Theta_{s+t}(x-b_n)$ is increasing for $b_n\leq x$.  Since
$x\geq b_m-2\sqrt{s+t}$, this sequence is increasing for $n^2\leq m^2-1$
and hence for $n<m$.  It is decreasing for
$n>m$.
The series
alternates and $a_n\Theta_{s+t}(x-b_n)$ is decreasing, so
\begin{eqnarray*}
\left|\sum_{n=m+1}^\infty (-1)^na_n\Theta_{s+t}(x-b_n)\right| & \leq & 
a_{m+1}\Theta_{s+t}(b_{m+1}-(b_m+2\sqrt{s+t}))\\
 & \leq & \frac{a_m}{2e^{4m^2}\sqrt{\pi(s+t)}}.
\end{eqnarray*}
Let $\norm{a}=\sup_{N\geq 1}\abs{\sum_{n=1}^N(-1)^na_n}$.
Using summation by parts and telescoping, we get 
\begin{align*}
&\left|\sum_{n=1}^{m-1} (-1)^na_n\Theta_{s+t}(x-b_n)\right|  \leq  
\left|\left(\sum_{n=1}^{m-1}(-1)^na_n\right)\Theta_{s+t}(x-b_{m-1})\right|\\
& \qquad\qquad\qquad  +\left|\sum_{n=1}^{m-2}\left(\sum_{k=1}^n (-1)^ka_k\right)\left[\Theta_{s+t}(x-b_{n+1})
-\Theta_{s+t}(x-b_n)\right]\right|\\
&\leq 2\norm{a}\Theta_{s+t}(b_m-2\sqrt{s+t}-b_{m-1})\\
&=\frac{2\norm{a}\exp(-4(m-1)^2)}{2\sqrt{\pi(s+t)}}
\leq \frac{a_m}{2e^2\sqrt{\pi(s+t)}}
\end{align*}
provided the inequality
\begin{equation}
4(m-1)^2\geq \log\left[2\norm{a}e^2\log(m+1)\right]\label{mlog}
\end{equation}
is satisfied.  Clearly we can take $M\geq 2$ large enough so that \eqref{mlog} holds for each $m\geq M$.  Then
\begin{align*}
&\intinf\abs{\fastx}^p\,dx  \geq  \sum_{m=M}^\infty\int_{\abs{x-b_m}<2\sqrt{s+t}}\abs{\fastx}^p\,dx\\
& \geq  \sum_{m=M}^\infty\int_{\abs{x-b_m}<2\sqrt{s+t}}\left|\frac{a_m}{2e\sqrt{\pi(s+t)}}
-\frac{a_m}{2e^{4m^2}\sqrt{\pi(s+t)}}-\frac{a_m}{2e^2\sqrt{\pi(s+t)}}\right|^p\,dx\\
& >  4\sqrt{s+t}\left[\frac{e-2}{2e^2\sqrt{\pi(s+t)}}\right]^p\sum_{m=M}^\infty a_m^p=\infty.
\end{align*}
(j) Since $F$ is bounded it satisfies the growth condition
$\abs{F(x)}\leq C_1\exp(C_2\abs{x}^{1+\alpha})$ ($C_1>0$, $C_2\in \R$,
$0\leq\alpha< 1$) which ensures $F\ast\Theta_t(x)$ is real analytic,
separately in $x\in\R$ and $t>0$ 
\cite[Theorem~10.2.1]{cannon}, \cite[Theorem~10.3.1]{cannon}.  
But then $(F\ast\Theta_t)'=f\ast\Theta_t$
is also analytic.
\end{proof}

\begin{remark}[Theorem~\ref{theoremheatc}]\label{remarktheoremheatc}
Part (a) shows
that $f\ast\Theta_t$ can be written as an improper Riemann integral.

Many results continue to hold in the space 
${\mathcal A}_{buc}:=\{f=F'\mid F\in{\mathcal B}_{buc}\}$
where ${\mathcal B}_{buc}$ are the functions in $C(\R)$ that
are  bounded and
uniformly
continuous.  The
multipliers for ${\mathcal A}_{buc}$ are the functions of bounded variation
that have limit $0$ at $\pm\infty$, which includes the heat kernel.
The space ${\mathcal A}_{buc}$ is a subspace of the weighted spaces
studied in Section~\ref{sectionweighted} below.  See 
Example~\ref{examplesalexc}(d).

When $f\in L^p$
for some $1\leq p\leq \infty$
the uniform estimate from the H\"older inequality is 
$$
\norm{f\ast\Theta_t}_\infty\leq
c_p\norm{f}_pt^{-1/(2p)}, \quad c_p=\left\{\!\!\!\begin{array}{cl}
1/(2\sqrt{\pi}), & p=1\\
(4\pi)^{-1/(2p)}[(p-1)/p]^{(p-1)/(2p)}, & 1<p<\infty\\
1, & p=\infty.
\end{array}
\right.
$$

For example, \cite{gustafson}.  The condition for equality in the H\"older inequality
\cite[p.~46]{liebloss}
shows these estimates are sharp.  When $1<p<\infty$ we can take $f=\Theta_t^{q/p}$.
When $p=\infty$ we can take $f=1$.  When $p=1$ the condition is 
$\Theta_t(x)=d\,{\rm sgn}[f(x)]$ for some $d\in\R$.  This cannot be satisfied with any
$f\in L^1$.  Instead,  we
take for a delta sequence.
Note that when $f\in L^1$ we get
the same estimate as for $f\in\alexc$ (Theorem~\ref{theoremheatc}(c)), with the Alexiewicz
norm replaced with the $L^1$ norm.

Part (h) of Theorem~\ref{theoremheatc} is given for $f\in L^1$ in \cite[p.~220]{epstein}.

Note that since $\fast$ is continuous the integral $\intinf \fast$
exists as a Henstock--Kurzweil integral and as an improper Riemann
integral but from (i) $\intinf\abs{\fastx}\,dx$ can diverge.

Theorem~10.2.1 in \cite{cannon} shows $f\ast\Theta_t(x)$ is separately
analytic in
$x\in\C$ and in a domain with bounded imaginary part of $t$.
\end{remark}

We have 
continuity with respect to initial
conditions.
\begin{corollary}\label{corollaryctsic}
Suppose $f,g\in\alexc$. 
\begin{enumerate}
\item[(a)] Then $\norm{f\ast\Theta_t-g\ast\Theta_t}
\leq \norm{f-g}$ for each $t>0$.
\item[(b)] Let $\epsilon>0$.  Suppose $u, v:\R\times(0,\infty)\to\R$ such that
$\norm{u(\cdot, t)-f}\to 0$ and $\norm{v(\cdot,t)-g}\to 0$ as
$t\to 0^+$. If $\norm{f-g}<\epsilon$ then for small enough $t$ we have
$\norm{u(\cdot, t)-v(\cdot,t)}<2\epsilon$.
\end{enumerate}
\end{corollary}
\begin{proof}
Part (a) follows from (f) and part (b) is a consequence of the triangle
inequality.
\end{proof}
Note that in (b) $u$ and $v$ need not be solutions of the heat equation
and need not be convolutions of the heat kernel with $f$ or $g$.
See \cite{gustafson} for corresponding results in $L^p$ spaces.

\begin{remark}[The regulated primitive integral and $L^p$ primitive integral]\label{remarkrpi}
A related integral is obtained by taking left continuous regulated functions as primitives.
A function is regulated on $\R$ if it has left and right limits at each point of $\R$ and it is
left continuous on $\Rbar$ if it equals its left limit at each point of
$(-\infty,\infty]$ and it has a limit at $-\infty$.  If $F$ is such a function
then the regulated primitive integral of $F'$ is $\intinf F'=F(\infty)$.
There are four types of integral on finite
intervals: $\int_{[a,b]}f=F(b+)-F(a-)=F(b+)-F(a)$,
$\int_{(a,b]}f=F(b+)-F(a+)$,
$\int_{[a,b)}f=F(b-)-F(a-)=F(b)-F(a)$,
$\int_{(a,b)}f=F(b-)-F(a+)=F(b)-F(a+)$.  The space of integrable distributions is a Banach
space under the Alexiewicz norm.  This integral contains the continuous primitive integral.
As well, all finite signed Borel measures are integrable.  For example, the Dirac 
measure is the distributional derivative of the Heaviside step function $H=\chi_{(0,\infty]}$.
The regulated primitive integral is discussed in \cite{talvilaregulated}.
While many of the properties proven in Theorem~\ref{theoremheatc} continue to hold, a key
difference between the two integrals is that, 
for the regulated primitive integral, translation is not
continuous in the Alexiewicz norm.  A consequence of this is that
$f\ast\Theta_t$ need not converge to $f$ as $t\to 0^+$.  
For example, $\delta\ast\Theta_t(x)=\Theta_t(x)$ and
$\norm{\delta\ast\Theta_t-\delta}
\geq|\int_{(0,\infty)}(\Theta_t-\delta)|
=1/2\not\to 0$
as $t\to 0^+$.  For this reason, we do not consider the regulated primitive integral in
this paper.  In Section~\ref{sectionhigher} we define higher order Alexiewicz spaces which
contain some linear combinations of translated Dirac distributions.
In all these spaces
we get convolutions with the initial data converging to the initial data 
in the appropriate norm.

We do have continuity in the $L^p$ norms for $1\leq p<\infty$.  Distributions that are the
distributional derivative of an $L^p$ function are considered in \cite{talvilaLp}.  When
these are convolved with the heat kernel they give solutions that take on initial conditions
in spaces of distributions that are isometrically isomorphic to $L^p$.
\end{remark}

When $f$ is locally a continuous function, $\fastx$ converges pointwise
to $f(x)$ as $t\to 0^+$.
\begin{theorem}\label{theoremfcts}
Suppose $f\in\alexc$ with primitive $F\in\balexc$ such that ${F|_I}\in
C^1(I)$  for some open interval $I\subset\R$.  Define
$u\fn\R\times[0,\infty)\to\R$ by
$$
u(x,t)=\left\{\begin{array}{cl}
\fastx, & (x,t)\in\R\times(0,\infty)\\
f(x), & (x,t)\in I\times\{0\}\\
0, & (x,t)\in (\R\setminus I)\times\{0\}.
\end{array}
\right.
$$
Then $u$ is continuous on $[\R\times(0,\infty)]\cup[I\times\{0\}]$.
\end{theorem}
\begin{proof}
Continuity of $u$ on $\R\times(0,\infty)$ is proved in 
Theorem~\ref{theoremheatc}(b).
Let $x_0\in I$.
Given $\epsilon>0$ there exists $\delta>0$ such that $(x_0-2\delta,
x_0+2\delta)\subset I$ and if
$\abs{x-x_0}<2\delta$ then $\abs{f(x)-f(x_0)}<\epsilon$.  For
$(x,t)\in(x_0-\delta, x_0+\delta)\times(0,\infty)$ we have
\begin{eqnarray*}
\abs{\fastx-f(x_0)} & = &  \left|\intinf f(\xi)\Theta_t(x-\xi)\,d\xi
-f(x_0)\intinf\Theta_t(x-\xi)\,d\xi\right|\\
 & \leq &  I_1+I_2+I_3,
\end{eqnarray*}
where
\begin{eqnarray*}
I_1 & = & \left|\int_{x_0-2\delta}^{x_0+2\delta}[f(\xi)-f(x_0)]\Theta_t(x-\xi)
\,d\xi\right|\leq\epsilon\intinf\Theta_t(x-\xi)\,d\xi=\epsilon,\\
I_2 & = &
\left|\int_{\abs{\xi-x_0}>2\delta}f(\xi)
\Theta_t(x-\xi)\,d\xi\right| \\
 & \leq & \norm{f}\left\{
V_{\xi\in\R}[\Theta_t(x-\xi)\chi_{(x_0+2\delta,\infty)}(\xi)]+
V_{\xi\in\R}[\Theta_t(x-\xi)\chi_{(-\infty,x_0-2\delta)}(\xi)]\right\}\\
 & \leq & 2\norm{f}\left[\Theta_t(\delta)+\Theta_t(-\delta)\right]\\
I_3 & = & \abs{f(x_0)}\int_{\abs{\xi-x_0}>2\delta}\Theta_t(x-\xi)\,d\xi\\
 & \leq  & \frac{\abs{f(x_0)}}{\sqrt{\pi}}\left(\int_{-\infty}^{-\delta/
(2\sqrt{t})}e^{-s^2}\,ds+\int_{\delta/
(2\sqrt{t})}^\infty e^{-s^2}\,ds\right).
\end{eqnarray*}
The estimate for $I_2$ comes from the H\"older inequality \eqref{holder}.
Letting $t\to 0^+$ shows $I_2\to 0$ and $I_3\to 0$.
\end{proof}

\begin{examples}\label{examplesalexc}
(a) $L^1$, the spaces of Henstock--Kurzweil integrable functions and
wide Denjoy integrable functions are all (non-closed) subspaces of
$\alexc$.  In each case, $\fastx$ agrees with its traditional value
and its value in $\alexc$.

(b) Let $V\in\balexc$ be a singular function, i.e., the pointwise 
derivative $V'(x)=0$ for almost all $x$.  Then $V'(x)\in L^1$ and the
Lebesgue integral gives $V'\ast\Theta_t(x)=0$ for each $x\in\R$.
This gives the zero solution of the heat equation with zero initial
condition taken on in the $L^1$ norm.
As an integral in $\alexc$ we get $V'\ast\Theta_t(x)=V\ast\Theta'_t(x)
=\intinf V(\xi)\Theta'_t(x-\xi)\,d\xi$.  This last exists as an
improper Riemann integral and gives a nontrivial solution of the
heat equation, initial values being taken on in the Alexiewicz norm.

(c) Suppose $f\in\alexc$ has compact support in the interval $[A,B]$.
Let $x>\max(B,0)$.  
Then $\fastx=1/(2\sqrtpt)\int_A^B f(\xi) \exp(-(\xi-x)^2/(4t))\,d\xi$.
By H\"older's inequality \eqref{holder},
$$
\abs{\fastx}  \leq  \frac{\norm{f}}{2\sqrtpt}V(\Theta_t\chi_{[A,B]})
 =  \frac{\norm{f}}{2\sqrtpt}\left[e^{-(x-A)^2/(4t)}
+e^{-(x-B)^2/(4t)}\right].
$$
Hence, for fixed $t>0$, $\fastx =O(\exp(-(x-B)^2/(4t))$ as $x\to\infty$.

(d) Suppose $f=F'$ where $F$ is bounded and uniformly continuous on $\R$
(but not necessarily on $\Rbar$).
We can define $\norm{f}=\sup_I\abs{\int_If}$, provided the supremum
is taken over bounded intervals $I\subset\R$.  The estimates in
(c) and (f) of Theorem~\ref{theoremheatc} then continue to hold,
even though $f$ need not be in $\alexc$.
For
example, let $f(x)=\sin(sx)$ for $s>0$.  A routine calculation
(also see Example~\ref{exampleweighted}(c) below) gives the
separated solution $u(x,t)=\fastx=\sin(sx)\exp(-s^2t)$.  Note
that $f\not\in\alexc$ but $\norm{f}=(1/s)\int_0^\pi\sin(x)\,dx=2/s$.
And, $\norm{u_t}_\infty=\exp(-s^2t)$.
A calculation gives $2\sqrt{\pi t}\,\norm{u_t}_\infty/\norm{f}
=s\sqrt{\pi t}\exp(-s^2t)\leq \sqrt{\pi/(2e)}<1$, in
accordance with (c) of Theorem~\ref{theoremheatc}.  And,
$\norm{u_t}=\exp(-s^2t)\norm{f}\leq\norm{f}$, in
accordance with (f) of Theorem~\ref{theoremheatc}. 
As well, $\norm{u_t-f}=(1-\exp(-s^2t))\norm{f}\to 0$ as $t\to 0^+$.

(e) Let $f(x)=\exp(ix^2/(4s))$ for $s>0$.  Then $f$ is bounded and
continuous but not in the $L^p$ spaces ($1\leq p<\infty$).  And,
$f$ is improper Riemann integrable so
$f\in\alexc$.  Completing the square gives
\begin{eqnarray*}
\fastx & = & \frac{1}{2\sqrtpt}\intinf e^{-(\xi-x)^2/(4t)} e^{i\xi^2/(4s)}d\xi\\
 & = & \frac{e^{\frac{-x^2}{4t}} e^{\frac{sx^2(s+it)}{4t(s^2+t^2)}}}{2\sqrtpt}
\intinf e^{\frac{-(s-it)}{4st}\left[\xi-\frac{s(s+it)x}{s^2+t^2}\right]^2}
d\xi.
\end{eqnarray*}
Due to Cauchy's theorem, the contour can be shifted in the complex plane
from $\xi$ real to ${\mathcal Im}(\xi) =stx/(s^2+t^2)$.  The integral
$\intinf e^{-a\xi^2}\,d\xi=\sqrt{\pi/a}$ for ${\mathcal Re}(a)>0$ then
gives
\begin{equation}
\fastx =\frac{e^{-\frac{tx^2}{4(s^2+t^2)}} e^{\frac{isx^2}{4(s^2+t^2)}}
\sqrt{s}\sqrt{s+it}}{\sqrt{s^2+t^2}}.\label{example(d)}
\end{equation}
It is readily seen that $\lim_{t\to 0^+}\fastx=f(x)$ at each point
$x\in\R$ (cf. Theorem~\ref{theoremfcts}).  
Differentiating $f$ and $f\ast\Theta_t$ with respect to $s$ gives an
example for which the initial data is not in $\alexc$.  See Example~\ref{exampleweighted}(e)
below.  

Note that 
for fixed $s,t>0$ there is a positive constant $c$ such that
$\fastx =O(e^{-cx^2})$ as $\abs{x}\to \infty$.  The
next example shows that $\fastx$ can tend to $0$ more slowly
(cf. Theorem~\ref{theoremheatc}(e)).

(f)  Suppose a decay rate $\omega\fn(0,\infty)\to(0,\infty)$ is given
and $\lim_{x\to\infty}\omega(x)=0$.
We will show there is $f\in L^p$ for each $1\leq p\leq\infty$
such that $\fastx\not=o(\omega(x))$
as $x\to\infty$.  We can assume $\omega$ is decreasing; otherwise
replace $\omega$ with $\tilde{\omega}(x)=\sup_{t>x}\omega(t)$.
Let $f(x)=\sum_{n=1}^\infty n^{-2}\Theta_s(x-b_n)$ where $s>0$
and
$\{b_n\}$ is an increasing sequence with limit $\infty$.  
Then $f\in L^p$.  And,
$f\ast\Theta_t(b_m)\geq m^{-2}\Theta_{s+t}(0)=[2m^2\sqrt{\pi(s+t)}]^{-1}$.
Let $b_n=\omega^{-1}(n^{-2})$.  Then
$b_n\uparrow\infty$.  And, 
$$
\frac{f\ast\Theta_t(b_m)}{\omega(b_m)}\geq \frac{1}{2\sqrt{\pi(s+t)}}
\not\to 0\quad\text{ as } m\to\infty.
$$

\end{examples}

We say $f\in\alexc$ has an absolutely convergent integral if $f$ has a 
primitive $F\in\bv\cap\balexc$.
\begin{prop}\label{propabsolutelyconv}
\begin{enumerate}
\item[(a)] Let $f\in\alexc$ with primitive $F\in\bv\cap\balexc$.  Then
$V(f\ast\Theta_t)\leq V\!F\, V\Theta_t=V\!F/\sqrt{\pi t}$.
\item[(b)] Let  $f\in\alexc$ with compact support in $[A,B]$.  Then
$$
V(f\ast\Theta_t)\leq
\frac{\norm{f}}{\sqrt{\pi}}\left[\frac{1}{\sqrt{t}}
+\frac{(B-A)\sqrt{2}}{\sqrt{e}\,t}\right].
$$
\item[(c)]  If $f\in L^1$ then
$V(f\ast\Theta_t)\leq\norm{f}_1/\sqrt{\pi t}$.
\item[(d)] Let $f\in\alexc$ and let $u(x,t)=u_t(x)=\fastx$.
For
$t>0$ the pointwise
derivatives have estimates
$\norm{u'_t}\leq\norm{f}V\Theta_t=\norm{f}/\sqrt{\pi t}$,
$\norm{\partial u(\cdot, t)/\partial t}=
\norm{u''_t}\leq\norm{f}V\Theta_t'=\norm{f}\sqrt{2}/(\sqrt{\pi e}\,t)$,
$\norm{u'_t}_\infty\leq\norm{f}V\Theta'_t=\norm{f}\sqrt{2}/(\sqrt{\pi e}\,t)$,
$\norm{\partial u(\cdot, t)/\partial t}_\infty
=\norm{u''_t}_\infty\leq\norm{f}V\Theta_t''=\norm{f}(1+4e^{-3/2})/(2\sqrt{\pi}\,
t^{3/2})$.
\end{enumerate}
\end{prop}
\begin{proof}
(a) Let $(x_i,y_i)$ be a sequence of disjoint intervals.
Integrating by parts
and using the Fubini--Tonelli theorem
\begin{eqnarray*}
\sum_i\left|f\ast\Theta_t(x_i)-f\ast\Theta_t(y_i)\right| & = & 
\sum_i
\left|\intinf\Theta_t'(\xi)\left[F(x_i-\xi)-F(y_i-\xi)\right]\,d\xi\right|\\
 & \leq & \intinf\abs{\Theta_t'(\xi)}\sum_i 
\left|F(x_i-\xi)-F(y_i-\xi)\right|\,d\xi\\
 & \leq & V\Theta_t\,V\!F.
\end{eqnarray*}
Taking the supremum over all such intervals gives the first result.
(b) Let the primitive of $f$ be $F$.  Then $F=0$ on $(-\infty,A]$ and
$F=F(B)$ on $[B,\infty)$. Integrate by parts and use the Fubini--Tonelli theorem in \cite{talvilaconv} to get
\begin{eqnarray*}
V(f\ast\Theta_t) & = & \intinf\abs{(f\ast\Theta_t)'(x)}\,dx\\
 & = & \intinf\left|F(B)\Theta_t'(x-B)+\int_A^B F(\xi)\Theta_t''(x-\xi)\,d\xi\right| dx\\
 & \leq & \abs{F(B)}/(\sqrt{\pi t})+ (B-A)\norm{F}_\infty\,V\Theta_t'\\
 & \leq & \frac{\norm{f}}{\sqrt{\pi}}\left[\frac{1}{\sqrt{t}}
+\frac{(B-A)\sqrt{2}}{\sqrt{e}\,t}\right].
\end{eqnarray*}
(c) This follows since if $f\in L^1$ then it has a primitive
that is absolutely continuous and of bounded variation.
(d) The heat kernel is smooth so
we can differentiate pointwise \cite[Corollary~4.5]{talvilaconv} to get 
$\partial u(x,t)/\partial x=u'_t(x)=f\ast\Theta'_t(x)$,
$\partial^2 u(x,t)/\partial x^2=u''_t(x)=\partial u(x,t)/\partial t=
f\ast\Theta''_t(x)$.  A calculation shows
$V\Theta_t'=4\Theta_t'(-\sqrt{2t})=\sqrt{2}/(\sqrt{\pi e}\,t)$ and 
$V\Theta_t''=4\Theta_t''(\sqrt{6t})-2\Theta_t''(0)=
(1+4e^{-3/2})/(2\sqrt{\pi}
\,t^{3/2})$.  The result then follows from the estimates
$\norm{f\ast g}\leq \norm{f}\norm{g}_1$ (\cite[Theorem~3.4]{talvilaconv})
and the H\"older inequality \eqref{holder}.
\end{proof}

\section{Higher order Alexiewicz spaces}\label{sectionhigher}

For each $n\in\N$ the set of distributions
that are the $n$th derivative of a function in $\balexc$ is a Banach
space isometrically isomorphic to $\balexc$.  A distributional integral
is defined, the multipliers being functions that are $n$-fold iterated
integrals of functions of bounded variation.  Properties analogous to Theorem~\ref{theoremheatc}
hold for these higher order distributions.  First we briefly introduce the
higher order distributional integrals.  More details can be found in \cite{talvilaacrn}.
Our presentation is simplified since we take primitives to be continuous rather
than regulated functions, c.f. Remark~\ref{remarkrpi}.

If $F,G\in\balexc$ and satisfy the distributional differential equation $F^{(n)}=G^{(n)}$ then
$F$ and $G$ differ by a polynomial of degree at most $n-1$.  But the only polynomial in $\balexc$
is the zero function so $F=G$.  This shows that according to the following definition the 
primitive of a distribution in $\acn$ is unique.

\begin{defn}\label{defnacnorm}
For each $n\in\N$ let
$\acn=\{f\in\Dp\mid f=F^{(n)} \text{ for some } F\in \balexc\}$, with $\acany{1}:=\alexc$.
If $f\in\acn$ define $\acnorm{f}=\norm{F'}=\sup_{x<y}\abs{F(x)-F(y)}$ where $F\in\balexc$
is the unique primitive of $f$.
\end{defn}

It is clear that $\acnorm{\cdot}$ is a norm on $\acn$ and this makes $\acn$ into a 
Banach space
that is isometrically isomorphic to $\balexc$.

\begin{defn}
Let $\nbv$ be the functions in $\bv$ that are right continuous on $\Rbar$, i.e.,
$g\in\bv$ and $\lim_{y\to x^+}g(y)=g(x)$ for each $x\in(-\infty,\infty)$.
Define $\Ibv^0=\nbv$.  Suppose $\Ibv^{n-1}$ is known for $n\in\N$.
Define $\Ibvn=\{h\fn\R\to\R\mid
h(x)=\int_0^xq(t)\,dt \text{ for some } q\in\Ibv^{n-1}\}$.
\end{defn}
This inductive definition defines the multipliers for integration in $\acn$.  Each function
in $h\in\Ibvn$ is an $n$-fold iterated integral of a unique function $g\in\nbv$.  We write this
as $h=I^{n-1}[g]$.
See \cite{talvilaacrn}, Proposition~2.5 and the
remark following.  Note that $h^{(k)}(0)=0$ for $0\leq k\leq n-2$.  Now we can define
the integral of the product of a distribution in $\acn$ and a function in 
$\Ibv^{n-1}$.

\begin{defn}\label{defnintacn}
Let $n\in\N$.  For $f\in\acn$ let $F$ be its primitive in $\balexc$.
For $h\in\Ibv^{n-1}$ such that $h=I^{n-1}[g]$ for $g\in\nbv$, define
the continuous primitive integral of $f$ with respect to $h$ as
\begin{eqnarray}
\intinf fh & = & \intinf F^{(n)}h=(-1)^{n-1}\intinf F'h^{(n-1)}
\label{intdefn1}\\
 & = & (-1)^{n-1}F(\infty)g(\infty)-(-1)^{n-1}\intinf F(x)\,dg(x).
\label{intdefn2}
\end{eqnarray}
\end{defn}
Note that $F'\in\alexc$ and $h^{(n-1)}\in\bv$ so that the last integral in \eqref{intdefn1}
is integration in $\alexc$ as in Section~\ref{sectiondistributionalintegrals}.  It is shown
in \cite[Proposition~2.10]{talvilaacrn} that if $P$ is a polynomial of degree at most $n-2$
then $\intinf fP=0$.  Hence, the iterated integrals defining $h$ can have any point $a\in\Rbar$
as their lower limit and $h^{(k)}$ would vanish at $a$ for $0\leq k\leq n-2$ and not necessarily at $0$.

The H\"older inequality now takes the following form.
\begin{prop}\label{holderacn}
Let $f\in\acn$ with primitive $F\in\balexc$.  Let $h\in\Ibv^{n-1}$ such that
$h=I^{n-1}[g]$ for $g\in\nbv$.  Then 
$$
\left|\intinf fh\right|\leq \left|\intinf F'\right|\inf\abs{g}+
\norm{F'} Vg\leq\acnorm{f}[\inf\abs{h^{(n-1)}}+Vh^{(n-1)}].
$$
\end{prop}

Every derivative of the heat kernel is a function of bounded variation vanishing at $\pm\infty$.
Hence, if $f\in\acn$ then $f\ast\Theta_t$ exists.  This leads to similar results as in
Theorem~\ref{theoremheatc}.
\begin{lemma}\label{lemmakernelvariation}
Let $n\in\N$.  Then $\Theta^{(n)}_t(x)=(-2\sqrt{t})^{-n}
\Theta_t(x)H_n(x/(2\sqrt{t}))$ and
$
V\Theta^{(n-1)}_t=\norm{\Theta^{(n)}_t}_1\leq c_nt^{-n/2}$
where
$$
c_n=\frac{1}{2^n\sqrt{\pi}}
\intinf e^{-x^2}\abs{H_n(x)}\,dx\leq k\sqrt{n!}\,2^{(1-n)/2}
$$
and $H_n$ is the $n$th order Hermite polynomial.
It is known
that $k<1.087$.
\end{lemma}
\begin{proof}
This follows from the Rodrigues formula for Hermite polynomials 
\cite[8.950.1]{gradshteyn} and Cramer's inequality for Hermite polynomials
\cite[8.954.2]{gradshteyn}.
\end{proof}
\begin{theorem}\label{theoremheatacn}
Let $n\geq 2$
and let $f\in\acn$.  Let the primitive of $f$ be $F\in\balexc$.
\begin{enumerate}
\item[(a)] The integrals 
$f\ast\Theta_t(x)=\Theta_t\ast f(x)=F\ast\Theta_t^{(n)}(x)
=[F\ast\Theta_t]^{(n)}(x)$ 
exist for
each $x\in\R$ and $t>0$.  
\item[(b)] $f\ast\Theta_t(x)$ is $C^\infty$ for $(x,t)\in\R\times(0,\infty)$.
\item[(c)] Let $c_n$ be as in Lemma~\ref{lemmakernelvariation} then
$\norm{f\ast\Theta_t}_\infty\leq c_n\acnorm{f} t^{-n/2}$.
The exponent on $t$ cannot be changed.
\item[(d)] Define the linear operator $\Phi_t\fn\acn\to C(\Rbar)$ by
$\Phi_t(f)=f\ast\Theta_t$.  Then $\norm{\Phi_t}\leq c_n t^{-n/2}$.
\item[(e)] $\lim_{\abs{x}\to\infty}\fastx=0$.
\item[(f)] For each $t>0$, $f\ast\Theta_t\in\acn$ and the inequality 
$\acnorm{f\ast\Theta_t}\leq\acnorm{f}$ is sharp
in the sense that the coefficient of $\acnorm{f}$ cannot be reduced.
Define the linear operator $\Psi_t\fn\acn\to \acn$ by
$\Psi_t(f)=f\ast\Theta_t$.  Then $\norm{\Psi_t}=1$.
For each integer $0\leq k\leq n$, $f\ast\Theta_t\in\alexc^{k}$
and $\norm{f\ast\Theta_t}^{(k)}\leq \norm{F'}c_{n-k+1}t^{(n-k+1)/2}$.
\item[(g)] Let $u(x,t)=f\ast\Theta_t(x)$ then $u$ is a solution of 
\eqref{heatclassic}-\eqref{heatpde} and
$\acnorm{u_t-f}\to 0$ as $t\to0^+$.
\item[(h)] For each $t>0$ we have $\intinf \fast=0$.
\item[(i)] For each $t>0$, $f\ast\theta_t(x)$ is real analytic as a function of 
$x\in\R$.  For each $x\in\R$, $f\ast\Theta_t(x)$ is real analytic as a 
function of $t>0$.
\end{enumerate}
\end{theorem}
\begin{proof}
From \eqref{intdefn1} we get $f\ast\Theta_t(x)=F'\ast\Theta_t^{(n-1)}(x)$.  This
then reduces the convolution of $f$ with $\Theta_t$ in $\acn$ to convolution of $F'$ with $\Theta_t^{(n-1)}$
in $\alexc$.  Now using Definition~\ref{defnacnorm}, the corresponding results of 
Theorem~\ref{theoremheatc} apply.  
(a) A linear change of variables theorem
is proved in \cite[Theorem~3.4]{talvilaacrn}.  This shows the convolution commutes.
(c) We have 
$$
\norm{f\ast\Theta_t}_\infty=\norm{F\ast\Theta_t^{(n)}}_\infty =
\sup_{x\in\R}\left|\intinf F(\xi)\Theta_t^{(n)}(x-\xi)\,d\xi\right|\leq\acnorm{f}\norm{\Theta_t^{(n)}}_1.
$$
Now use Lemma~\ref{lemmakernelvariation}.
Let $s>0$ and take $f(x)=\Theta_s^{(n-1)}(x)$.  Then 
$$
f\ast\Theta_t(x)=[\Theta_{s+t}]^{(n-1)}(x) =\frac{(-1)^{n-1}\Theta_{s+t}(x)H_{n-1}(x/(2\sqrt{s+t}))}{2^{n-1}
(s+t)^{(n-1)/2}}.
$$
Let $\alpha_n\in\R$ such that $H_{n-1}(\alpha_n)\not=0$.  Then 
$$
\norm{f\ast\Theta_t}_\infty\geq\abs{f\ast\Theta_t(2\alpha_n\sqrt{s+t})}=
\frac{\exp(-\alpha_n^2)H_{n-1}(\alpha_n)}{2^{n}\sqrt{\pi}\,(s+t)^{n/2}}.
$$ 
Letting $s\to 0$ shows
the exponent on $t$ cannot be changed.
(e)  Using Lemma~\ref{lemmakernelvariation}, dominated
convergence gives
$$
\lim_{x\to\pm\infty}f\ast\Theta_t(x)  =  \lim_{x\to\pm\infty}
\intinf F(x-\xi)\Theta_t^{(n)}(\xi)\,d\xi
  =  F(\mp\infty)\intinf\Theta_t^{(n)}(\xi)\,d\xi=0.
$$
(f) To show the inequality is sharp let $f=\Theta_s^{(n-1)}$.
Note that for each integer $k\leq n$ we have 
$F\ast\Theta_t^{(n-k)}\in C(\Rbar)$.
And, $f\ast\Theta_t=(F\ast\Theta_t^{n-k})^{(k)}\in\alexc^{k}$.
(h) Similar analysis as in the proof of Theorem~\ref{theoremheatc} shows
\begin{eqnarray*}
\int_\alpha^\beta \fastx\,dx & = & 
\intinf F'(\xi)\int_\alpha^\beta\Theta_t^{(n-1)}(x-\xi)
\,dx\,d\xi\\
 & = & \intinf \left[F(\beta-\xi)
-
F(\alpha-\xi)\right]\Theta_t^{(n-1)}(\xi)\,d\xi.
\end{eqnarray*}
Dominated convergence then gives 
$$
\intinf f\ast\Theta_t(x)\,dx  = 
F(\infty)\intinf \Theta_t^{(n-1)}(\xi)\,d\xi
  = 0.
$$
\end{proof}
\begin{remark}[Theorem~\ref{theoremheatacn}\label{remarkacn}]
(c) The sharp constant is not known for $n\geq 2$.
\end{remark}
\begin{examples}\label{exampleacn}
(a) Let $w$ be a bounded continuous function such that the pointwise derivative $w'(x)$ exists for no $x$.
Define $F_1(x)=w(x)\exp(-\abs{x})$ then $F_1\in\balexc$.  For $n\geq 1$, $F_1^{(n)}$ has pointwise values 
nowhere  and yet $F_1^{(n)}\ast\Theta_t$ is well-defined as an integral in $\acn$ and defines a smooth
solution of the heat equation, taking on initial 
values in the norm $\acnorm{\cdot}$, as in Theorem~\ref{theoremheatacn}.
Note that as a Lebesgue integral $F_1^{(n)}\ast\Theta_t$ does not exist.

(b) Neither the Dirac distribution nor any of its derivatives are in $\acn$ 
\cite[Proposition~3.1]{talvilaacrn}.  However, some linear combinations of the Dirac distribution are
in $\acn$.  Define $F_2(x)=0$ for $x\leq 0$, $F_2(x)=x$ for $0\leq x\leq 1$, $F_2(x)=1$ for $x\geq 1$.  Then
$F_2\in\balexc$.  Let $\delta_a=\tau_a\delta$, the Dirac distribution supported at $a\in\R$, with 
$\delta_0=\delta$.  Then $F_2'(x)=H(x)-H(x-1)$ where $H$ is the Heaviside step function.  For $n\geq 2$
we get $F_2^{(n)}=\delta^{(n-2)}-\delta_1^{(n-2)}\in\acn$.  A solution of the heat
equation is then given using Lemma~\ref{lemmakernelvariation}
\begin{eqnarray*}
u(x,t) & = & F_2^{(n)}\ast\Theta_t(x)=\Theta_t^{(n-2)}(x)-\Theta_t^{(n-2)}(x-1)\\
 & = & (-2\sqrt{t})^{-(n-2)}\left[\Theta_t(x)H_{n-2}\left(\frac{x}{2\sqrt{t}}\right)
-\Theta_t(x-1)H_{n-2}\left(\frac{x-1}{2\sqrt{t}}\right)\right].
\end{eqnarray*}
And, $\norm{u_t-F_2^{(n)}}^{(n)}\to 0$ as $t\to 0^+$.

(c) Fix $\alpha >0$.  Let $p_\alpha(x)=x^\alpha$ and $q(x)=\exp(-x)$.  Define $F_3=H p_\alpha q\in\balexc$.
The pointwise derivative satisfies $F_3^{(n)}(x)\sim\alpha(\alpha-1)\cdots(\alpha-n+1)x^{\alpha-n}$ as $x\to 0^+$.
Hence, for each $n\in\N$ we have $F_3^{(n)}\in \acn$ and for $n\geq\alpha+1$ we have 
$F_3^{(n)}\not\in L^1_{loc}$.  
Although $F_3^{(n)}$ need not be locally integrable in the Lebesgue (or Henstock--Kurzweil) sense,
$F_3^{(n)}\ast\Theta_t$ gives a 
smooth
solution of the heat equation, taking on initial 
values in the norm $\acnorm{\cdot}$, as in Theorem~\ref{theoremheatacn}.
\end{examples}

\begin{prop}\label{proputestimatesalexn}
Let $c_n$ be as in Lemma~\ref{lemmakernelvariation}.
\begin{enumerate}
\item[(a)] Let $f\in\acn$ with primitive $F\in\bv\cap\balexc$.  Then
$V(f\ast\Theta_t)\leq V\!F\, \norm{\Theta^{(n)}_t}_1=c_nV\!F\,t^{-n/2}$.
\item[(b)] Let  $f\in\acn$ with compact support in $[A,B]$.  Then
$V(f\ast\Theta_t)\leq\newline\acnorm{f}[c_nt^{-n/2}+(B-A)c_{n+1}\,t^{-(n+1)/2}$.
\item[(c)]  If $f=F^{(n)}$ where $F\in AC\cap\bv$ then
$V(f\ast\Theta_t)\leq c_{n}\norm{F'}_1\,t^{-n/2}$.
\item[(d)] Let $f\in\acn$ and let $u(x,t)=u_t(x)=\fastx$.
For
$t>0$ the pointwise
derivatives have estimates
$\acnorm{u'_t}\leq\acnorm{f}V\Theta_t=\acnorm{f}/\sqrt{\pi t}$,
$\acnorm{\partial u(\cdot, t)/\partial t}=
\acnorm{u''_t}\leq\acnorm{f}V\Theta_t'=\acnorm{f}\sqrt{2}/(\sqrt{\pi e}\,t)$,
$\norm{u'_t}_\infty\leq\acnorm{f}V\Theta^{(n)}_t=c_{n+1}
\acnorm{f}t^{-(n+1)/2}$,
$\norm{\partial u(\cdot, t)/\partial t}_\infty
=
\norm{u''_t}_\infty\leq\newline\acnorm{f}V\Theta^{(n+1)}_t=c_{n+2}\acnorm{f}
t^{-(n+2)/2}$.
\end{enumerate}
\end{prop}
\begin{proof}
Using Theorem~\ref{theoremheatacn}, the proof is similar to the proof of
Proposition~\ref{propabsolutelyconv}.
\end{proof}

\section{Weighted spaces}\label{sectionweighted}

The minimal condition for existence of the convolution
 $\fastx$ is that the integral
$\intinf f(x)e^{-x^2/(4\tau)}\,dx$ exist for some $\tau>0$.  Then
$\fastx$ exists for $0<t<\tau$.

Let $G\in\balexc$.  Let $\tau>0$ and define $\omega_\tau(x)=\exp(-x^2/(4\tau))$.
Let $F_\tau$ be the unique continuous solution of the
Volterra integral equation
\begin{equation}
G(x)-G(0)=F_\tau(x)\omega_\tau(x)+\frac{1}{2\tau}\int_0^x F_\tau(\xi)\xi 
\omega_\tau(\xi)\,d\xi.\label{integraleqn}
\end{equation}
Necessarily, $F_\tau(0)=0$.
Since the kernel is continuous there is a unique
solution for $F_\tau\omega_\tau$ and hence for $F_\tau$.
Uniqueness is proven in \cite{yosidaode}.
If $G_i\in\balexc$ and 
$$
G_i(x)-G_i(0)=F_\tau(x)\omega_\tau(x)+\frac{1}{2\tau}\int_0^x F_\tau(\xi)\xi 
\omega_\tau(\xi)\,d\xi
$$
for $i=1,2$ then $G_1(x)-G_2(x)=G_1(0)-G_2(0)$ so $G_1-G_2\in\balexc$ but
is constant so $G_1=G_2$.  Hence, there is a one-to-one correspondence between
$F_\tau$ and $G$ in \eqref{integraleqn}.  We then have the following definitions
for a 
family of 
weighted Alexiewicz spaces.
\begin{defn}\label{defnalext}
Let $\tau>0$.  Define $\omega_\tau\fn\R\to\R$ by $\omega_\tau(x)=e^{-x^2/(4\tau)}$.
Let $\balexct=\{F_\tau\in C(\R)\mid F_\tau \text{ is a solution of } \eqref{integraleqn}
\text{ for some } G\in\balexc\}$.  For $F_\tau\in\balexct$ define 
$\norm{F_\tau}_{\tau,\infty}=\norm{G}'_\infty$ where
$G$ is the unique function in \eqref{integraleqn}.  Define
$
\alexct=\{f\in\Dp\mid f=F_\tau' \text{ for some } F_\tau\in\balexct\}$.
For $f\in\alexct$ define $\actnorm{f}=\norm{F_\tau}_{\tau,\infty}$ where 
$F_\tau$ is
the unique primitive of $f$.
\end{defn}
For the definition to be meaningful we still need to show distributions in
$\alexct$ have a unique primitive in $\balexct$.  Suppose $f\in\alexct$ and
$f=F_1'=F_2'$ for $F_1, F_2\in\balexct$.  As before, $F_1-F_2=c=\text{constant}$.
There are unique $G_i\in\balexc$ such that 
$$
G_i(x)-G_i(0)=F_i(x)\omega_\tau(x)+\frac{1}{2\tau}\int_0^x F_i(\xi)\xi 
\omega_\tau(\xi)\,d\xi
$$ 
for $i=1,2$.  Integrating shows $G_1(x)-G_2(x)-G_1(0)+G_2(0)=c\omega_\tau(0)$.
Putting $x=0$ shows $c=0$.

The $\alexct$ spaces are Banach spaces under a weighted Alexiewicz norm.
A property
that makes them
somewhat delicate to work with is that
they are not closed under
translations.

\begin{theorem}\label{theoremalexct}
Let $\tau>0$. 
\begin{enumerate}
\item[(a)] $\balexct$ is a Banach space isometrically isomorphic to $\balexc$.
\item[(b)] $\alexct$ is a Banach space isometrically isomorphic to $\alexc$.
\item[(c)] $f\in\alexct$ if and only if $f\omega_\tau\in\alexc$.
\item[(d)] If $f\in\alexct$ then $\actnorm{f}=\norm{f\omega_\tau}$.
\item[(e)] For each $\tau>0$, $\alexc\subsetneq\alexct$.
\item[(f)] $0<r<s$ if and only if $\alexcany{s}\subsetneq\alexcany{r}$.
\item[(g)] If $f\in\alexct$ and $F(x)=\int_0^xf$ then $\abs{F(x)}\leq 
\actnorm{f}e^{x^2/(4\tau)}$ for all $x\in\R$.  And, $F(x)=o(e^{x^2/(4\tau)})$ as 
$\abs{x}\to \infty$.
\item[(h)] If $0<r<s$ then $\alexcany{s}$ is a dense subspace of $\alexcany{r}$.
\item[(i)] If $0<r<s$ then $\norm{f}_r\leq 2\norm{f}_s$ for each $f\in\alexcany{s}$.
\item[(j)] Define $L^1(\omega_\tau)=
\{f\in L^1_{loc}\mid f\omega_\tau\in L^1\}$.  Then $L^1(\omega_\tau)$ is dense
in $\alexcany{\tau}$.
\item[(k)] $\alexct$ is not closed under translation.
\end{enumerate}
\end{theorem}
\begin{proof}
(a) Using the uniqueness for Volterra integral equations, \eqref{integraleqn}
defines a linear isometry between $\balexc$ and $\balexct$.
(b) Follows from the uniqueness of primitives proven above.
(c) Let $F_\tau\in\balexct$.  Then $F_\tau\in C(\R)$ so $F'_\tau$ is 
integrable on compact intervals.  From \eqref{integraleqn} let $G$ be the corresponding function in $\balexc$.  We have
$$
G(x)-G(0)  = 
F_\tau(x)\omega_\tau(x)+\frac{1}{2\tau}\int_0^x F_\tau(\xi)\xi 
\omega_\tau(\xi)\,d\xi
  =  \int_0^x G'=
\int_0^x F'_\tau\omega_\tau.
$$
Therefore, $G'=F'_\tau\omega_\tau$ as elements in $\alexc$.  The isometry
between $\alexc$ and $\alexct$ is then given by $\alexct\ni f\leftrightarrow
f\omega_\tau\in\alexc$.  If $f\in\alexct$ then $f\omega_\tau\in\alexc$ with
primitive $F(x)=\int_{-\infty}^x f\omega_\tau$.
The product $f\omega_\tau$ is well-defined
via $\langle f\omega_\tau,\phi\rangle =\langle f,\phi\omega_\tau\rangle$
for each $\phi\in\D$ since the integral $\intinf (f\omega_\tau)\phi$ exists.
(d)
$$
\norm{f}_\tau=
\norm{F_\tau}_{\tau,\infty}  =  \sup_{x<y}\left|G(x)-G(y)\right|
  =  \sup_{x<y}\left|\int_x^y G'\right|
  =  \sup_{x<y}\left|\int_x^y F_\tau'\omega_\tau\right| =\norm{f\omega_\tau}.
$$
(e) Since $\omega_\tau\in\bv$, if $f\in\alexc$ then $f\omega_\tau\in\alexc$ so
$\alexc\subset\alexct$.
(f) Suppose $0<r<s$.  Let $f\in\alexcany{s}$.  Then $f\omega_s\in\alexc$.  And,
$f\omega_r=(f\omega_s)(\omega_r/\omega_s)$.  The function $\omega_r/\omega_s\in\bv$ so
$f\in\alexcany{r}$.  The examples given below show 
the set inclusions in (e) and (f) are proper.
Suppose $r>s>0$.  The function $f(x)=\exp(x^2/(4r))$ is in $\alexcany{s}$ but not
in $\alexcany{r}$.  The case $r=s$ is trivial.
(g)  Suppose $x>0$.  Write 
$\exp(-x^2/(4\tau))F(x)=\intinf\left[f(\xi)\omega_\tau(\xi)\right]g_x(\xi)\,d\xi$
where $g_x(\xi)=e^{(\xi^2-x^2)/(4\tau)}\chi_{[0,x]}(\xi)$.  Note that 
$\lim_{x\to\infty}g_x(\xi)=0$ for each $\xi\in\R$ and the variation of $g_x$
on $\R$ is
$1$.  The H\"older inequality
\eqref{holder}
gives the first estimate.  By Theorem~\ref{proplimit} we have
$F(x)=o(\exp(x^2/(4\tau))$ as $x\to\infty$.  Similarly as $x\to-\infty$.
(h) Given $f\in\alexcany{r}$ let $f_n=f\chi_{(-n,n)}$.  Then $\{f_n\}\subset\alexcany{s}$.
Let $\alpha<\beta$.  Then $f\chi_{(\alpha,\beta)}\in\alexcany{r}$ and
$f\omega_\tau\chi_{(\alpha,\beta)}\in\alexc$ so
$$
\int_\alpha^\beta\left[f-f_n\right]\omega_r  = 
\int_{-\infty}^{-n}f\omega_r\chi_{(\alpha,\beta)}
+\int_n^\infty f\omega_r\chi_{(\alpha,\beta)}
$$
and $\norm{f-f_n}_r\leq \norm{f\omega_r\chi_{(-\infty,-n)}}+
\norm{f\omega_r\chi_{(n,\infty)}}\to 0$ as $n\to\infty$.
(i) The H\"older inequality \eqref{holder} gives
$$
\left|\int_\alpha^\beta f\omega_r\right|  =  \left|\int_\alpha^\beta\left[f\omega_s\right]
\left[\frac{\omega_r}{\omega_s}\right]\right|\\
  \leq  2\norm{f}_s.
$$
(j) Let $\epsilon>0$.  Let $f\in\alexcany{\tau}$.  Define $F(x)=\int_{-\infty}^x
f\omega_\tau$.  Then $F\in\balexc$.  Note that the primitives of $L^1$ functions
are given by $AC\cap\bv$.  Since $L^1$ is dense in $\alexc$
(\cite[Proposition~3.3]{talvilaconv}) there is $G\in AC\cap\bv$ such that
$\norm{F-G}_\infty<\epsilon$.  Now let $g=G'/\omega_\tau$.
As $1/\omega_\tau$ is locally bounded we have $g\in L^1_{loc}$.  Hence, 
$g\in L^1(\omega_\tau)$.  And,
$
\norm{f-g}_\tau\leq 2\norm{F-G}_\infty<2\epsilon$.
(k) Let $f(x)=[(x^2+1)\omega_\tau(x)]^{-1}$.  Then $f\in\alexct$.  But
$f(x+1)\omega_\tau(x)=\exp(x/[2\tau])\exp(1/[4\tau])[(x+1)^2+1]^{-1}$ so
this translation is  not in $\alexct$.
\end{proof}
This shows that all of the spaces $\balexc$, $\alexc$, $\acn$, $\alexct$,
$\balexct$ are isometrically isomorphic.

If $\alpha>0$ then $f(x)=\exp(x^2/(4\alpha))\in\alexct$
if and only if $\alpha>\tau$.  This shows that the distributions in $\alexct$ need
not be tempered.  And, $g(x)=\exp(\beta\abs{x}^\gamma)\in\alexct$ for
all $\tau>0$ and all $\beta\in\R$ if and only if $0\leq\gamma<2$.
For each $\tau>0$ every polynomial is in $\alexct$.  The H\"older inequality shows
that $L^p\subset\alexct$ for each $1\leq p\leq\infty$.  Similarly for weighted
$L^p$  spaces.
Let $\sigma>0$ and $L^p(\omega_{\sigma})$ ($1\leq p<\infty$)
be the Lebesgue measurable functions for which 
$\intinf \abs{f(x)}^p\,\omega_{\sigma}(x)\,dx$ 
exists.
The H\"older inequality shows $L^p(\omega_\sigma)\subset\alexct$ for all 
$\tau<p\sigma$ when $1<p<\infty$ and $L^1(\omega_\sigma)\subset\alexct$ for all 
$\tau\leq\sigma$ .

Now we look at analogues of Theorem~\ref{theoremheatc} in weighted spaces.
Proofs are similar to the corresponding parts of that theorem
except as where noted.
\begin{theorem}\label{theoremheatweight}
Let $\tau>0$, $0<t<\tau$ and $f\in\alexct$.  Let the primitive of $f$ be $F\in C(\R)$.
\begin{enumerate}
\item[(a)] The integrals 
$f\ast\Theta_t(x)=\Theta_t\ast f(x)=F\ast\Theta_t'(x)
=[F\ast\Theta_t]'(x)$ 
exist for
each $x\in\R$ and $0<t<\tau$.  
\item[(b)] $f\ast\Theta_t(x)$ is $C^\infty$ for $(x,t)\in\R\times(0,\tau)$.
\item[(c)]  Let $0<t<\sigma\leq\tau$.  For each $x\in\R$, 
$\abs{f\ast\Theta_t(x)}\leq\norm{f}_\sigma\exp(x^2/[4(\sigma-t)])/(2\sqrt{\pi t})$.
Let $0<t\leq\tau-\sigma$.  The estimate $\norm{(f\ast\Theta_t)\omega_\sigma}_\infty
\leq\norm{f}_\tau/(2\sqrt{\pi t})$ is sharp in the sense that the coefficient
of $\norm{f}_\tau$ cannot be reduced.
\item[(d)] Let $\sigma\geq\tau$ and let $C(\omega_\sigma,\R)$ be the Banach space of 
continuous functions with the
norm $\norm{F\omega_\sigma}_\infty$.  Define the linear operator 
$\Phi_t\fn\alexct\to C(\omega_\sigma,\R)$ by
$\Phi_t(f)=f\ast\Theta_t$.  Then $\norm{\Phi_t}=1/(2\sqrtpt)$.
\item[(e)] $\fastx=o(\exp(x^2/(4[\tau-t]))$ as $\abs{x}\to\infty$.
\item[(f)] Let $0<\sigma<\tau$ and $0<t<\tau-\sigma$. Then $f\ast\Theta_t\in\alexcany{\sigma}$ and
$\norm{f\ast\Theta_t}_\sigma\leq\sqrt{(\tau-\sigma)/(\tau-\sigma-t)}\norm{f}_\sigma$.
Define the linear operator $\Psi_t\fn\alexcany{\sigma}\to \alexcany{\sigma}$ by
$\Psi_t(f)=f\ast\Theta_t$.  Then $\norm{\Psi_t}= \sqrt{(\tau-\sigma)/(\tau-\sigma-t)}$.
\item[(g)] Let $u(x,t)=f\ast\Theta_t(x)$ then $u\in C^2(\R)\times C^1((0,\tau))$ and $u$ is a solution 
of  the heat equation in this region.  Let $0<\sigma<\tau$.  Then
$\lim_{t\to 0^+}\norm{u_t-f}_\sigma=0$.
(h) Let $0<\sigma<\tau$ and $0<t<\tau-\sigma$.  Then $\intinf \fastx\Theta_\sigma(x)\,dx=\intinf f(x)
\Theta_{\sigma+t}(x)\,dx$.
\item[(i)] Let $0<t<\tau-\sigma$. Then
$\norm{f\ast\Theta_t}_\sigma\leq\sqrt{\sigma(\tau-t)/[t(\tau-\sigma-t)]}\actnorm{f}$.
Now consider
$\Psi_t\fn\alexct\to \alexcany{\sigma}$, as the restriction of the operator
defined in part (f) to the subspace $\alexct$.
Then $\norm{\Psi_t}\leq \sqrt{\sigma(\tau-t)/[t(\tau-\sigma-t)]}$.
\item[(j)] For each $0<t<\tau$, $f\ast\theta_t(x)$ is real analytic as a function of 
$x\in\R$.  For each $x\in\R$, $f\ast\Theta_t(x)$ is real analytic as a 
function of $0<t<\tau$.
\end{enumerate}
\end{theorem}
\begin{proof}
(a) This follows from the fact that for each $x\in\R$ and each $0<t<\tau$ the function
$\xi\mapsto \Theta_t(x-\xi)/\omega_\tau\in\bv$.  The estimate in
Theorem~\ref{theoremalexct}(g) allows integration by parts.  Taylor's
theorem shows that $[F\ast\Theta_t]'(x)=F'\ast\Theta_t(x)$.
(b) Similar to Theorem~\ref{theoremheatc}(b).
(c)  Write
$$
\fastx  = \intinf 
\left[f(\xi)\omega_\sigma(\xi)\right]\left[\Theta_t(x-\xi)/\omega_\sigma(\xi)\right]\,d\xi.
$$
The function $\xi\mapsto \Theta_t(x-\xi)/\omega_\sigma(\xi)$ is increasing on
$(-\infty,\sigma x/(\sigma-t)]$ and decreasing on $[\sigma x/(\sigma-t),\infty)$.  
Using the
second mean value theorem for integrals as in Theorem~\ref{theoremheatc}(c) we have,
\begin{equation}
\fastx=\frac{\Theta_t(x-\sigma x/(\sigma-t))}{\omega_\sigma(\sigma x/(\sigma-t))}\int_{x_1}^{x_2}
f\omega_\sigma=\frac{e^{\frac{x^2}{4(\sigma-t)}}}
{2\sqrt{\pi t}}\int_{x_1}^{x_2}
f\omega_\sigma,\label{2ndmvt}
\end{equation}
where now $x_1\leq\sigma x/(\sigma -t)$ and $x_2\geq \sigma x/(\sigma -t)$.
This gives the first inequality.  Also,
$\abs{\fastx\omega_\sigma(x)}\leq \exp(x^2/[4(\tau-t)])\exp(-x^2/[4\sigma])\norm{f}_\tau
/(2\sqrt{\pi t})$ and the second inequality follows.
Using \eqref{heatnorm}-\eqref{heatproduct}, 
the example $f=\Theta_s$ and the limit $s\to 0^+$ 
shows the estimate is sharp.
(d) This follows from (c).
(e) Write
$$
e^{-x^2/(4[\tau-t])}\fastx=\frac{1}{2\sqrt{\pi t}}\intinf\left[f(\xi)\omega_\tau(\xi)\right]
\left[e^{-(x-\xi)^2/(4t)}e^{\xi^2/(4\tau)}e^{-x^2/(4[\tau-t])}\right]\,d\xi.
$$
The function $\xi\mapsto\exp(-(x-\xi)^2/(4t))\exp(\xi^2/(4\tau))\exp(-x^2/(4[\tau-t]))$
has variation as a function of $\xi$ equal to $2$.  It has limit $0$ as $\abs{x}\to\infty$.
The result follows by Theorem~\ref{proplimit}.
(f)
Let $0<t<\tau-\sigma$ and $-\infty<\alpha<\beta
<\infty$.   From (a)
\begin{eqnarray}
\int_\alpha^\beta \fastx\omega_\sigma(x)\,dx & = & \int_\alpha^\beta\intinf f(x-\xi)
\Theta_t(\xi)\omega_\sigma(x)\,d\xi\,dx\label{g1}\\
 & = & \int_\alpha^\beta\intinf f(\xi)\Theta_t(x-\xi)\omega_\sigma(x)\,d\xi\,dx\label{g2}\\
 & = & \intinf[f(\xi)\omega_\tau(\xi)]\int_\alpha^\beta
\frac{\Theta_t(x-\xi)\omega_\sigma(x)}{\omega_\tau(\xi)}\,dx\,d\xi.\label{g3}
\end{eqnarray}
We can interchange orders of integration in \eqref{g3} since the function
$\xi\mapsto\Theta_t(x-\xi)\omega_\sigma(x)/\omega_\tau(\xi)$ has variation
$\exp(x^2/[4(\tau-t)])\exp(-x^2/[4\sigma])/\sqrt{\pi t}$.  The conditions
of \cite[Proposition~A.3]{talvilaconv} are then satisfied.  Now consider
\begin{align}
&\intinf\int_\alpha^\beta f(x-\xi)\Theta_t(\xi)\omega_\sigma(x)\,dx\,d\xi\\
&\quad=\intinf\Theta_t(\xi)\intinf[f(x)\omega_\tau(x)]\left[\frac{\omega_\sigma(x+\xi)}
{\omega_\tau(x)}\right]\chi_{(\alpha-\xi,\beta-\xi)}(x)\,dx\,d\xi.
\label{g4}
\end{align}
The change of variables is justified by \cite[Theorem~11]{talviladenjoy}.
Note that $V_{x\in\R}[\omega_\sigma(x+\xi)/\omega_\tau(x)]=\exp(\xi^2/[4(\tau-\sigma)])$.
Use the inequality $V(gh)\leq V\!g\norm{h}_\infty+\norm{g}_\infty Vh$ with $g(x)
=\omega_\sigma(x+\xi)/\omega_\tau(x)$ and $h(x)=\chi_{(\alpha-\xi,\beta-\xi)}(x)$ to see that
the conditions of \cite[Proposition~A.3]{talvilaconv} are satisfied.   We then have equality
between \eqref{g1} and \eqref{g4}, upon interchanging orders of
integration in \eqref{g4} and changing variables.  This gives
$$
\int_\alpha^\beta\fastx\omega_\sigma(x)\,dx  =  
\intinf\Theta_t(\xi)\intinf
\left[f(x-\xi)\omega_\sigma(x-\xi)\chi_{(\alpha,\beta)}(x)\right]\psi(x)\,dx\,d\xi,
$$
where $\psi(x)=\omega_\sigma(x)/\omega_\tau(x-\xi)$.
Note that $\psi$ has a maximum at $-\sigma\xi/(\tau-\sigma)$.  Using
the second mean value theorem as in Theorem~\ref{theoremheatc}(c), there are
$x_1\leq -\sigma\xi/(\tau-\sigma)\leq x_2$ such that
\begin{align}
&\left|\int_\alpha^\beta\fastx\omega_\sigma(x)\,dx\right|
\label{prop3.7}\\
&\qquad  =  
\left|\intinf\Theta_t(\xi) e^{\xi^2/[4(\tau-\sigma)]}\int_{(x_1,x_2)\cap(\alpha,\beta)}
f(x-\xi)\omega_\sigma(x-\xi)\,dx\,d\xi\right|\notag\\
&\qquad \leq  \norm{f}_\sigma\intinf\Theta_t(\xi) e^{\xi^2/[4(\tau-\sigma)]}\,d\xi
 =  \norm{f}_\sigma\sqrt{\frac{\tau-\sigma}{\tau-\sigma-t}}.\notag
\end{align}
This also shows the limits exist as $\alpha\to-\infty$ and $\beta\to
\infty$.  Hence, $f\ast\Theta_t\in\alexcany{\sigma}$.
If we let $f=\Theta_{-s}$ for $s>\tau$ we get $f\ast\Theta_t=\Theta_{t-s}$,
$\norm{f}_\sigma=\sqrt{\sigma/(s-\sigma)}$ and
$\norm{f\ast\Theta_t}_\sigma=\sqrt{\sigma/(s-\sigma-t)}$.  Letting $s\to\tau^+$
shows the estimate is sharp and gives the norm of $\Psi_t$.
(g) Since we can differentiate under the integral sign, $u$ is a solution of
the heat equation in $\R\times(0,\tau)$.
To show the initial conditions are
taken on in the weighted norm,
use the equality between \eqref{g1} and \eqref{g4}.  This gives
$$
\int_\alpha^\beta[\fastx-f(x)]\omega_\sigma(x)\,dx=\intinf\Theta_t(\xi)\int_\alpha^\beta
[f(x-\xi)-f(x)]\omega_\sigma(x)\,dx\,d\xi.
$$
Write
$$
\int_\alpha^\beta
[f(x-\xi)-f(x)]\omega_\sigma(x)\,dx =I_1(\xi)-I_2(\xi)+I_3(\xi)
$$
where 
\begin{eqnarray*}
I_1(\xi) & = & \int_{\alpha-\xi}^\alpha f(x)\omega_\sigma(x+\xi)\,dx\\
I_2(\xi) & = & \int_{\beta-\xi}^\beta f(x)\omega_\sigma(x+\xi)\,dx\\
I_3(\xi) & = & \int_{\alpha}^\beta f(x)[\omega_\sigma(x+\xi)-\omega_\sigma(x)]\,dx.
\end{eqnarray*}
Now show integrals of  $I_1$, $I_2$, $I_3$ against $\Theta_t(\xi)$ 
tend to $0$ as $t\to 0^+$, uniformly in $\alpha$
and $\beta$.  Let $F_\tau(x)=\int_{-\infty}^x f\omega_\tau$.
Use the H\"older inequality \eqref{holder} to get
\begin{eqnarray*}
\intinf\Theta_t(\xi)\abs{I_1}(\xi)\,d\xi & \leq & 2\intinf\Theta_t(\xi)
\sup_{\alpha\in\R}\max_{y,z\in[\alpha-\xi,\alpha]}\abs{F_\tau(y)-
F_\tau(z)}
e^{\xi^2/[4(\tau-\sigma)]}\,d\xi\\
 & \to 0 & \text{ as }t\to 0^+
\end{eqnarray*}
since $F_\tau$ is uniformly continuous on $\R$ and $\Theta_t$ is a
summability kernel (approximate identity).  Similarly, for $\xi<0$.  
Similarly, for $I_2$.  And,
$$
\abs{I_3}(\xi)\leq 2\norm{F_\tau}_\infty\intinf\left|\frac{\partial}{\partial x}
\frac{\omega_\sigma(x+\xi)-\omega_\sigma(x)}{\omega_\tau(x)}\right|\,dx.
$$
Dominated convergence allows us to take the limit $\xi\to 0$ under the integral
sign to get $I_3(\xi)\to 0$.  Using \eqref{g1}, \eqref{g4}, \eqref{heatnorm} we have
$$
\norm{f\ast\Theta_t-f}_\sigma\leq \intinf\Theta_t(\xi)\norm{f(\cdot-\xi)-f(\cdot)}_\sigma
\,d\xi
$$
and $\norm{f(\cdot-\xi)-f(\cdot)}_\sigma$ is continuous at $\xi=0$.  But, $\Theta_t$ is
a summability kernel so $\lim_{t\to 0^+}\norm{u_t-f}_\sigma=0$.
(h) The calculation following \eqref{g3} shows that 
$$
\intinf\fastx\Theta_\sigma(x)\,dx= \intinf f(\xi)\Theta_t\ast\Theta_\sigma(\xi)\,d\xi
=\intinf f(\xi)\Theta_{\sigma+t}(\xi)\,d\xi.
$$
(i) We can write 
$$
\int_\alpha^\beta \fastx\omega_\sigma(x)\,dx=\int_\alpha^\beta\intinf \left[f(\xi)
\omega_\tau(\xi)\right]\left[\frac{\Theta_t(x-\xi)}{\omega_\tau(\xi)}\right]
d\xi\,\omega_\sigma(x)
\,dx.
$$
Now use \eqref{2ndmvt} to get
$$
\norm{f\ast\Theta_t}_\sigma  \leq  \frac{2\sqrt{\pi \sigma}\sqrt{\tau-t}\,\Theta_{t-\tau}\ast
\Theta_\sigma(0)\norm{f}_\tau}{\sqrt{t}}
=\sqrt{\frac{\sigma(\tau-t)}{t(\tau-\sigma-t)}}\,\norm{f}_\tau.
$$
(j) See Theorem~10.2.1,  Theorem~10.3.1 in \cite{cannon} and the proof of 
Theorem~\ref{theoremheatc}(j).  Part (g) of Theorem~\ref{theoremalexct} gives the necessary
growth condition on $F$.
\end{proof}

Note that in (f) the coefficient remains bounded as $t\to 0^+$ but not so in (i).
This will be important for uniqueness Theorem~\ref{theoremuniqueweighted} below.

\begin{examples}\label{exampleweighted}
(a) If $f\in\alexct$ then $f\ast\Theta_t$ need not be in $\alexct$.  Let 
$f(x)=1/[(x^2+1)\omega_\tau(x)]$.  Since $f\geq 0$, for each $t>0$ 
the Fubini--Tonelli theorem gives
\begin{eqnarray*}
\actnorm{f\ast\Theta_t} & = &  \intinf\intinf\frac{\Theta_t(x-\xi)\omega_\tau(x)}
{\omega_\tau(\xi)(\xi^2+1)}\,d\xi\,dx
= 2\sqrt{\pi \tau}\intinf\frac{e^{\xi^2/(4\tau)}}{\xi^2+1}\Theta_t\ast\Theta_\tau(\xi)
\,d\xi\\
 & = & 2\sqrt{\pi \tau}\intinf\frac{e^{\xi^2/(4\tau)}\Theta_{t+\tau}(\xi)}{\xi^2+1}
\,d\xi =\infty.
\end{eqnarray*}

(b) Let $s>\tau$ and take $f=\Theta_{-s}$.  Then $f\in\alexct$ and for $0<t<s$ 
formula \eqref{heatconvolution} gives $f\ast\Theta_t(x)=\Theta_{t-s}(x)=\exp(x^2/[4(s-t)])/[2\sqrt{\pi(s-t)}]$.
Note that if $0<\sigma<s$ then $f\ast\Theta_t\in\alexcany{\sigma}$ if and only if
$0<t<s-\sigma$.  In particular, $f\ast\Theta_t\in\alexct$ if and only if
$0<t<s-\tau$.

(c)  Let $z$ be a complex number.  If $f(x)=\exp(izx)$ then we get the plane wave
solution $\fastx=e^{izx-z^2t}$ (c.f. \cite[p.~207]{john}).  
For all complex $z$, both $f$ and $f\ast\Theta_t$ are in $\alexct$ for all
$\tau>0$.

(d)  If $f$ is a polynomial of degree $n$ then
$f\ast\Theta_t$ is a polynomial in $x$ and $t$ of degree $n$ in $x$ and degree $\lfloor
n/2\rfloor$ in $t$.  To see this let $p_n(x)=x^n$.  The binomial theorem followed by
basic gamma function identities then give
\begin{eqnarray*}
p_n\ast\Theta_t(x) & = & \frac{1}{\sqrt{\pi}}\sum_{m=0}^n(-1)^m\binom{n}{m}x^{n-m}
(2\sqrt{t})^m\intinf \xi^m e^{-\xi^2}\,d\xi\\
 & = & \sum_{\ell=0}^{\lfloor n/2\rfloor}\frac{n!x^{n-2\ell}t^\ell}{(n-2\ell)!\ell!}
 =  (i\sqrt{t})^n H_n(x/[2i\sqrt{t}]).
\end{eqnarray*}
The last line comes from the explicit form of the Hermite polynomial.  See
\cite[8.950.2]{gradshteyn}. For example,
$p_3\ast\Theta_t(x)=x^3+6xt$.
In the literature these are called heat polynomials \cite{widderbook}.
For each polynomial $f$ we have $f\ast\Theta_t\in\alexct$ for
each $\tau>0$.  This also gives
$H_n\ast\Theta_t(x)=(1-4t)^{n/2}H_n(x/\sqrt{1-4t})$.  The above quoted formula shows
this is valid for all $t>0$.

(e) If we take $g(x)=\partial \exp(ix^2/(4s))/\partial s =-ix^2\exp(ix^2/(4s))/(4s^2)$.
Then we can get $g\ast\Theta_t$ by differentiating the solution in \eqref{example(d)}.
Note that $g\in\alexct$ for each $\tau>0$.

(f) Let $p$ be a polynomial, $a>0$, $b\in\R$, $0<\gamma<1$.  Define
$f(x)=p(x)\exp(x^2/(4a)+b\abs{x}^\gamma)$.  Let $0<\tau<a$.  Writing
$G(x)=\int_{-\infty}^x f(\xi)\omega_{\tau}(\xi)\,d\xi$ 
defines $G\in\balexc$.
Since $G'=f\omega_\tau$ we have $f\in\alexct$.
\end{examples}

\begin{prop}\label{propweight}
Let $0<\sigma<\tau$ and $0<t<\tau-\sigma$.
Let $f\in\alexct$ and $u(x,t)=\fastx$.  Then
$\norm{u_t'}_\sigma\leq 
(\tau-\sigma)/[\sqrt{\pi t}\,(\tau-\sigma-t)]$;
$$
\norm{\partial u(\cdot, t)/\partial t}_\sigma
=\norm{u_t''}_\sigma
\leq 
\frac{\norm{f}_\sigma}{2t}\left[\frac{(\tau-\sigma)^{3/2}}{
(\tau-\sigma-t)^{3/2}}+\frac{\sqrt{\tau-\sigma}}{\sqrt{\tau-\sigma-t}}\right];
$$
$$
\abs{u_t'(x)}
\leq \frac{e^{x^2/4(\sigma-t)}\norm{f}_\sigma\sqrt{\sigma}}
{\sqrt{\sigma-t}}\left[
\frac{1}{t}+\frac{\abs{x}}{2\sqrt{\pi\sigma(\sigma-t)t}}\right].
$$
\end{prop}
\begin{proof}
As in \eqref{prop3.7},
$$
\left|\int_\alpha^\beta f\ast\Theta_t'(x)\omega_\sigma(x)\,dx\right|
  \leq  \norm{f}_\sigma\intinf\abs{\Theta_t'(\xi)}e^{\xi^2/[4(\tau
-\sigma)]}\,d\xi
=  \frac{(\tau-\sigma)\norm{f}_\sigma}{\sqrt{\pi t}\,(\tau-\sigma-t)}.
$$
Similarly,
\begin{eqnarray*}
\left|\int_\alpha^\beta f\ast\Theta_t''(x)\omega_\sigma(x)\,dx\right|
 & \leq & 
 \norm{f}_\sigma\intinf\abs{\Theta_t''(\xi)}e^{\xi^2/[4(\tau
-\sigma)]}\,d\xi\\
 & \leq & \frac{\norm{f}_\sigma}{2t}\intinf\Theta_t(\xi) 
e^{\xi^2/[4(\tau-\sigma)]}\left(\frac{\xi^2}{2t}+1\right)\,d\xi\\
 & = & \frac{\norm{f}_\sigma}{2t}\left[\frac{(\tau-\sigma)^{3/2}}{
(\tau-\sigma-t)^{3/2}}+\frac{\sqrt{\tau-\sigma}}{\sqrt{\tau-\sigma-t}}\right].
\end{eqnarray*}
As in the proof of Theorem~\ref{theoremheatweight}(c),
$$
\abs{u_t'(x)}  \leq  \norm{f}_\sigma V_{\xi\in\R}
\frac{\Theta_t'(x-\xi)}{\omega_\sigma(\xi)}.
$$
Note that
\begin{eqnarray*}
V_{\xi\in\R}
\frac{\Theta_t'(x-\xi)}{\omega_\sigma(\xi)}
 & = &  \frac{1}{4\sqrt{\pi}\,t^{3/2}}
\intinf e^{-(\xi-x)^2/(4t)}e^{\xi^2/(4\sigma)}\left|
\frac{(\xi-x)^2}{2t}-1-\frac{\xi(\xi-x)}{2\sigma}\right|\,d\xi\\
 & = & \frac{e^{x^2/(4[\sigma-t])}\sqrt{\sigma}}{2\sqrt{\pi(\sigma-t)}\,t}
\intinf e^{-s^2}\left|2s^2+xs\sqrt{\frac{t}{\sigma(\sigma-t)}} -1
\right|\,ds,
\end{eqnarray*}
upon completing the square.
\end{proof}
An inequality for
$\abs{\partial u(x, t)/\partial t}=\abs{u_t''(x)}$ can be proved in the
same manner.

\section{Uniqueness}\label{sectionuniq}
The example $u(x,t)=\delta'\ast\Theta_t(x)=\Theta_t'(x)=
-x\exp(-x^2/(4t))/(4\sqrt{\pi}\,t^{3/2})$ shows that the heat equation need not
have a unique solution, since $u(x,0)=0$ for all $x$.
Uniqueness can be obtained by imposing
a boundedness condition on norms of solutions.  We prove uniqueness
by reducing to one of the two classical theorems.
\begin{theorem}\label{theoremclassicaluniq1}
Let $u\in C^2(\R)\times C^1((0,\infty))$ such that $u \in C(\R\times[0,\infty))$,
$u_2-u_{11}=0$ in $\R\times(0,\infty)$, $u$ is bounded, $u(x,0)=f(x)$ for
a bounded continuous function $f\fn\R\to\R$. Then the unique solution is given by
$u_t(x)=\fastx$.
\end{theorem}
\begin{theorem}\label{theoremclassicaluniq2}
Let $\tau>0$.
Let $u\in C^2(\R)\times C^1((0,\tau))$ such that $u \in C(\R\times[0,\tau))$,
$u_2-u_{11}=0$ in $\R\times(0,\tau)$, $\abs{u(x,t)}\leq Ae^{Bx^2}$ for 
some constants $A$, $B$ and all $(x,t)\in\R\times[0,\tau)$, $u(x,0)=f(x)$ for
a continuous function $f\fn\R\to\R$ satisfying $\abs{f(x)}\leq Ae^{Bx^2}$. 
Then the unique solution is given by
$u_t(x)=\fastx$.
\end{theorem}
The second theorem is due to Tychonoff.  Widder provides a proof in
\cite{widder} and also discusses various
related results: 
T\"acklind's
generalisation of the allowed exponent and
a uniqueness
condition for non-negative solutions.
See also \cite{widderbook}.

For uniqueness in $\alexc$, $\acn$, or $\alexct$,
pointwise conditions are not applicable.  However, we obtain
a uniqueness condition by requiring boundedness of the solution
in the Alexiewicz norm.  We use a method similar to that used by
Hirschman and Widder in proving uniqueness under an $L^p$ norm condition 
\cite[Theorem~9.2]{hirschmanwidder}.
\begin{theorem}
[Uniqueness in $\alexc$]
\label{theoremuniquealex}
Let $u\in C^2(\R)\times C^1((0,\infty))$ such that
$u_2-u_{11}=0$ in $\R\times(0,\infty)$, $\norm{u_t}$ is bounded, 
$\lim_{t\to 0^+}\norm{u_t-f}=0$ for some $f\in\alexc$.
The unique solution is then given by
$u_t(x)=\fastx$.
\end{theorem}
\begin{proof}
From Theorem~\ref{theoremheatc}
we see that $f\ast\Theta_t$ satisfies the hypotheses.

If $u$ and $v$ are two solutions, let $h=u-v$.  Then  
$h\in C^2(\R)\times C^1((0,\infty))$, 
$h_2-h_{11}=0$ in $\R\times(0,\infty)$, $\norm{h_t}$ is bounded, 
$\lim_{t\to 0^+}\norm{h_t}=0$.  Let $y,t>0$ and let
$\psi_y(x,t)=(2y)^{-1}\int_{x-y}^{x+y}h_t(\xi)\,d\xi$.  Then
$$
\frac{\partial \psi_y(x,t)}{\partial t}  =  
\frac{1}{2y}\int_{x-y}^{x+y}\frac{\partial h(\xi,t)}{\partial t}\,d\xi
  =  \frac{1}{2y}\int_{x-y}^{x+y}h_t''(\xi)\,d\xi
  =  \frac{h_t'(x+y)-h_t'(x-y)}{2y}.
$$
And,
$$
\frac{\partial^2 \psi_y(x,t)}{\partial x^2}  =
\frac{\partial}{\partial x}\frac{h(x+y,t)-h(x-y,t)}{2y}  =
\frac{h_1(x+y,t)-h_1(x-y,t)}{2y}.
$$
Hence, $\psi_y\in C^2(\R)\times C^1((0,\infty))$ and is a solution of the heat equation.
Note that $\abs{\psi_y(x,t)}\leq \norm{h_t}/(2y)$.  It then follows
that $\psi_y$ is bounded for $(x,t)\in\R\times(0,\infty)$
and tends to $0$ as $t\to0^+$. Define $\psi_y(x,0)=0$.  Let $x\in\R$ and 
$(\alpha,\beta)\in
\R\times(0,\infty)$.  Then $\abs{\psi_y(x,0)-\psi_y(\alpha,\beta)}\leq
\norm{h_\beta}/(2y)\to 0$ as $(\alpha,\beta)\to (x,0)$.  Hence, $\psi_y\in
C(\R\times[0,\infty))$.  By Theorem~\ref{theoremclassicaluniq1} we have
$\psi_y=0$ for each $y>0$.  By the continuity of $h$ we get
$\lim_{y\to 0^+}\psi_y(x,t)=0=h(x,t)$.
\end{proof}
\begin{lemma}\label{lemmauniqueacn}
Let $n\in\N$.  Let $g_n(x)=\sum_{k=0}^n\binom{n}{k}(-1)^k\chi_{(a_k,a_{k+1})}(x)$,
where $a_k=k/(n+1)$.  Define
$G_n(x)=\int_0^x\cdots\int_0^{x_{i+1}}\cdots\int_0^{x_2}g_n(x_1)\,dx_1\ldots
dx_i\ldots dx_n$.  Then $G_n\in C^{n-1}([0,1])$, $G_n(x)=O(x^n)$ as $x\to 0^+$,
$G_n(x)=O((x-1)^{n})$ as $x\to 1^-$ and $G_n>0$ on $(0,1)$.
\end{lemma}
\begin{proof}
The definition shows $G_n\in C^{n-1}([0,1])$.  

We have $G_n(x)=[(n-1)!]^{-1}\int_0^x(x-x_1)^{n-1}g_n(x_1)dx_1$.
To prove that $G_n>0$ it suffices to evaluate at each $a_k$.  Hence,
\begin{eqnarray*}
G_n(a_{l+1}) & = & \frac{1}{(n-1)!}\sum_{k=0}^l\binom{n}{k}(-1)^k\int_{a_k}^{a_{k+1}}
(a_{l+1}-x_1)^{n-1}\,dx_1\\
 & = & \frac{1}{n!(n+1)^n}\sum_{k=0}^l\binom{n+1}{k}(-1)^k(l+1-k)^n\\
 & = & \frac{A(n,l)}{n!(n+1)^n},
\end{eqnarray*}
where $A(n,l)$ is the Eulerian number of the first kind.  Combinatoric arguments
show these are positive for $0\leq l<n$ and that $A(n,n)=0$.  See, for example,
\cite{grahamknuth}.   

It is clear that
$$
\lim_{x\to 0^+}\frac{G_n(x)}{x^n}=\lim_{x\to 0^+}\frac{G_n^{(n)}(x)}{n!}=
\frac{g_n(0^+)}{n!}=\frac{1}{n!}.
$$
Note that $1-y\in(a_k,a_{k+1})$ if and only if $y\in(a_{n-k},a_{n-k+1})$.  So,
$g_n(1-y)=(-1)^ng_n(y)$.  A change of variables then shows
\begin{equation}
G_n(1-x)=G_n(x)+\frac{(-1)^n}{(n-1)!}\int_0^1(y-x)^{n-1}g_n(y)\,dy.\label{G_na}
\end{equation}
And,
\begin{eqnarray}
\int_0^1(y-x)^{n-1}g_n(y)\,dy & = & \sum_{k=0}^n\binom{n}{k}(-1)^k\int_{a_k}^{a_{k+1}}(y-x)^{n-1}\,dy
\notag\\
 & = & \frac{(-1)^{n+1}}{n(n+1)^n}\sum_{k=0}^{n+1}\binom{n+1}{k}(-1)^k\left[(n+1)x-k\right]^n\notag\\
 & = & -\frac{1}{n}\sum_{m=0}^n\binom{n}{m}\frac{(-x)^{n-m}}{(n+1)^m}T_{n+1,m},\label{G_nc}
\end{eqnarray}
where $T_{n,m}=\sum_{k=0}^{n}\binom{n}{k}
(-1)^kk^m$.  (The term $k^m$ is defined to be $1$ when $k=m=0$.)
  Let $S_n(x)=(1-x)^n=\sum_{k=0}^n\binom{n}{k}(-1)^kx^k$.  Then 
$$
S_n^{(m)}(1)=0=\sum_{k=0}^n\binom{n}{k}(-1)^kk(k-1)\cdots(k-m+1)
$$ for
$0\leq m\leq n-1$.  It follows that $T_{n,m}$ is a linear combination of
$S_n^{(p)}(1)$ for $0\leq p\leq m$.  Hence, $T_{n,m}=0$ ($0\leq m\leq n-1$).  
From \eqref{G_na} and \eqref{G_nc} we see that
$G_n(1-x)=G_n(x)$.  This gives
the order relation as $x\to 1^-$.

\end{proof}

\begin{theorem}[Uniqueness in $\acn$]
\label{theoremuniqueacn}
Let $n\geq 2$.
Let $u\in C^2(\R)\times C^1((0,\infty))$ such that
$u_2-u_{11}=0$ in $\R\times(0,\infty)$, $\acnorm{u_t}$ is bounded, 
$\lim_{t\to 0^+}\acnorm{u_t-f}=0$ for some $f\in\acn$.
The unique solution is then given by
$u_t(x)=\fastx$.
\end{theorem}
\begin{proof}
Theorem~\ref{theoremheatacn}(a),(b),(f),(g)
show that $f\ast\Theta_t$ satisfies the
hypotheses.  The other part of the proof is similar to that of
Theorem~\ref{theoremuniquealex}.  Use the same notation for $h_t$.
From Lemma~\ref{lemmauniqueacn}, the function
$\xi\mapsto G_{n-1}((\xi-x+y)/(2y))\chi_{(x-y,x+y)}(\xi)$ is in
$\Ibv^{n-1}$ (Definition~\ref{defnacnorm}).  Since $h_t\in\acn$ the integral 
\begin{eqnarray*}
\psi_y(x,t) & = & \frac{1}{2y}\intinf
G_{n-1}\left(\frac{\xi}{2y}-\frac{x-y}{2y}\right)\chi_{(x-y,x+y)}(\xi)h_t(\xi)\,d\xi\\
 & = & \frac{1}{2y}\int_{x-y}^{x+y}
G_{n-1}\left(\frac{\xi}{2y}-\frac{x-y}{2y}\right)h_t(\xi)\,d\xi
\end{eqnarray*}
exists. Integrating by parts gives
\begin{eqnarray*}
\frac{\partial\psi_y(x,t)}{\partial t} & = & 
\frac{1}{2y}\int_{x-y}^{x+y}
G_{n-1}\left(\frac{\xi}{2y}-\frac{x-y}{2y}\right)
\frac{\partial h(\xi,t)}{\partial t}\,d\xi\\
 & = & -\frac{1}{(2y)^{2}}\int_{x-y}^{x+y}
G_{n-1}'\left(\frac{\xi}{2y}-\frac{x-y}{2y}\right)
h_t'(\xi)\,d\xi.
\end{eqnarray*}
And,
\begin{eqnarray*}
\frac{\partial^2\psi_y(x,t)}{\partial x^2} & = & 
\frac{1}{(2y)^3}\int_{x-y}^{x+y}
G''_{n-1}\left(\frac{\xi}{2y}-\frac{x-y}{2y}\right)
h_t(\xi)\,d\xi\\
 & = &-\frac{1}{(2y)^{2}}\int_{x-y}^{x+y}
G_{n-1}'\left(\frac{\xi}{2y}-\frac{x-y}{2y}\right)
h_t'(\xi)\,d\xi, 
\end{eqnarray*}
so that $\psi_y$ is a solution of the heat equation. From Proposition~\ref{holderacn},
$$
\abs{\psi_y(x,t)}  \leq  \frac{\norm{h_t}^{(n)}}{(2y)^n}\left[
\linf_{[0,1]}\abs{g_{n-1}}+V_{[0,1]}g_{n-1}\right]
  =  \frac{\norm{h_t}^{(n)}}{y^n}
\to 0 \text{ as } t\to 0^+.
$$
For fixed $y>0$, $\psi_y$ can then be continuously extended to vanish on $t=0$.
Theorem~\ref{theoremclassicaluniq1} shows $\psi_y=0$. Let
$(x,t)\in\R\times(0,\infty)$.  Since $h_t$ is continuous and $G_{n-1}>0$ in
$(0,1)$ the mean value theorem for integrals shows there is $x^\ast\in(x-y,x+y)$
such that
$$\lim_{y\to 0^+}\psi_y(x,t)  = 0=  \lim_{y\to 0^+}h_t(x^\ast)\int_0^1
G_{n-1}(\xi)\,d\xi=h_t(x)\int_0^1G_{n-1}(\xi)\,d\xi.
$$
It now follows that $h_t(x)=0$.
\end{proof}

\begin{theorem}[Uniqueness in $\alexct$]
\label{theoremuniqueweighted}
Let $\tau>0$.
Let $u\in C^2(\R)\times C^1((0,\tau))$ such that
$u_2-u_{11}=0$ in $\R\times(0,\tau)$.  Suppose that for each $0<\sigma<\tau$
the quantity $\norm{u_t}_\sigma\sqrt{\tau-\sigma-t}$ 
is bounded for all
$0<t<\tau-\sigma$ and
$\lim_{t\to 0^+}\norm{u_t-f}_\sigma=0$ for some $f\in\alexct$.
The unique solution is then given by
$u_t(x)=\fastx$.
\end{theorem}
\begin{proof}
Theorem~\ref{theoremheatweight}(a),(b),(f),(g)
show that $f\ast\Theta_t$ satisfies the
hypotheses.  The other part of the proof is similar to that of
Theorem~\ref{theoremuniquealex}. 
Use the same notation.  As before, $h$ and $\psi_y$  are solutions of the heat 
equation
in $\R\times(0,\tau)$.
Now use the H\"older inequality \eqref{holder} to get
$$
\abs{\psi_y(x,t)}=\frac{1}{2y}\left|\int_{x-y}^{x+y}\left[h_t(\xi)\omega_\sigma(\xi)\right]
\,\frac{d\xi}{\omega_\sigma(\xi)}\right|
\leq \frac{\norm{h_t}_\sigma}{y}\left[e^{(x+y)^2/(4\sigma)}+e^{(x-y)^2/(4\sigma)}\right].
$$
If $0<y\leq 1$ then $\abs{\psi_y(x,t)}\leq 2\norm{h_t}_\sigma\exp(1/\sigma)\exp(x^2/\sigma)/y$.
Fix $0<\rho<\tau$ and let $0<\sigma<\rho$.
We then have $A= 2\sup_{0<t\leq\tau-\rho}\norm{h_t}_\sigma\exp(1/\sigma)/y$ and $B=1/\sigma$ in 
Theorem~\ref{theoremclassicaluniq2}.
Define $\psi_y(x,0)=0$.
Let $x\in\R$ and 
$(\alpha,\beta)\in
\R\times(0,\tau)$.  Then $\abs{\psi_y(x,0)-\psi_y(\alpha,\beta)}\leq
2\norm{h_\beta}_\sigma\exp(1/\sigma)\exp(\alpha^2/\sigma)/y\to 0$ as $(\alpha,\beta)\to (x,0)$.
Hence, $\psi_y\in
C(\R\times[0,\tau))$.  By Theorem~\ref{theoremclassicaluniq2} we have
$\psi_y=0$ on $\R\times[0,\tau-\rho)$ for each $y>0$.
By the continuity of $h$ we get
$\lim_{y\to 0^+}\psi_y(x,t)=0=h(x,t)$ for each 
$(x,t)\in\R\times(0,\tau-\rho)$.  But, $0<\rho<\sigma<\tau$ were arbitrary
so we have $h=0$ on $\R\times(0,\tau)$.
\end{proof}

\end{document}